\documentclass[bj, preprint]{imsart}

\usepackage{amsthm,amsmath}
\RequirePackage[colorlinks,citecolor=blue,urlcolor=blue]{hyperref}

\startlocaldefs

\usepackage{enumerate}
\usepackage{amssymb,amsfonts}
\usepackage{bbm}

\numberwithin{equation}{section}
\theoremstyle{plain}
\newtheorem{theorem}{Theorem}[section]
\newtheorem{lemma}[theorem]{Lemma}
\newtheorem{proposition}[theorem]{Proposition}
\newtheorem{corollary}[theorem]{Corollary}

\theoremstyle{remark}
\newtheorem{remark}[theorem]{Remark}
\newtheorem{definition}[theorem]{Definition}
\newtheorem{example}[theorem]{Example}

\newcommand{\bra}[1]{\left( #1 \right)}
\newcommand{\sqa}[1]{\left[ #1 \right]}
\newcommand{\cur}[1]{\left\{ #1 \right\}}
\newcommand{\ang}[1]{\left< #1 \right>}
\newcommand{\abs}[1]{\left| #1 \right|}
\newcommand{\nor}[1]{\left\| #1 \right\|}
\newcommand{\norm}[1]{\left\| #1 \right\|}

\newcommand{\rnum}{\mathbb{R}}

\newcommand{\prob}{\mathbb{P}}
\newcommand{\R}{\mathbb{R}}

\renewcommand{\d}{\, \mathrm{d}}

\newcommand{\ca}{\mathbf{a}}
\newcommand{\meas}{\mathcal{M}}

\newcommand{\cA}{\mathcal{A}}
\newcommand{\cF}{\mathcal{F}}

\newcommand{\Prob}{\mathcal{P}_2}

\newcommand{\ACQV}[1]{\operatorname{AC}^{#1}}

\newcommand{\cL}{\mathcal{L}}
\newcommand{\cV}{\mathcal{V}}
\newcommand{\cE}{\mathcal{E}}

\newcommand{\costM}{\cc_{BB}}
\newcommand{\costFP}{\cc_{FPE}}

\newcommand{\Sim}{\operatorname{Sym}}
\newcommand{\esp}[1]{\mathbb{E}\sqa{#1}}

\newcommand{\cc}{{\textsf{c}}}
\renewcommand{\chi}{\mathbf{1}}
\newcommand{\eps}{\varepsilon}
\renewcommand{\epsilon}{\varepsilon}

\DeclareMathOperator{\id}{Id}

\newcommand{\E}{\mathbb{E}}
\renewcommand{\P}{\mathbb{P}}

\newcommand{\conv}{*}

\newcommand{\F}{\mathcal{F}}

\renewcommand{\subset}{\subseteq}

\newcommand{\qv}[1]{\langle {#1}\rangle}

\newcommand{\Wd}{\cc_{MOT} }

\DeclareMathOperator{\Tr}{Tr}

\endlocaldefs

\renewcommand{\varphi}{\phi}
\usepackage{dsfont}
\renewcommand{\chi}{\mathds{1}}
\begin{document}

\begin{frontmatter}

\title{A Benamou-Brenier formulation of martingale optimal transport}

\runtitle{A Benamou-Brenier formulation of martingale optimal transport}

\begin{aug}

\author{\fnms{Martin} \snm{Huesmann} \thanksref{a}\ead[label=e1]{huesmann@iam.uni-bonn.de}}
\and
\author{\fnms{Dario} \snm{Trevisan} \thanksref{b}\corref{}\ead[label=e2]{dario.trevisan@unipi.it}}
\address[a]{Institut f\"ur Angewandte Mathematik, Rheinische Friedrich-Wilhelms-Universit\"at Bonn,\\
DE-53115 Bonn, Germany. \printead{e1}}
\address[b]{Dipartimento di Matematica, Universit\`a degli Studi di Pisa,\\
56127 Pisa, Italy. {\printead{e2}}}



\end{aug}

\begin{abstract}
 We introduce a Benamou-Brenier formulation for the continuous-time martingale optimal transport problem as a weak length relaxation of its discrete-time counterpart. By the correspondence between  classical martingale problems  and Fokker-Planck equations, we obtain an equivalent PDE formulation for which  basic properties such as existence, duality and geodesic equations can be analytically studied, yielding corresponding results for the stochastic formulation.  In the one dimensional case, sufficient conditions for finiteness of the cost are also given and a link between geodesics and porous medium equations is partially investigated.
\end{abstract}




\end{frontmatter}

\section{Introduction}

Given two probability measures $\mu, \nu$ on $\rnum^d$ and a cost function $\cc:\R^d\times\R^d\to\R$, the martingale optimal transport problem is the variational problem
\begin{equation}\label{eq:discrete mot} 
\inf \cur{ \E\sqa{ \cc( X_0, X_1  )}  :  \text{ $(X_0,X_1)$ is a  martingale, $(X_0)_\sharp \prob = \mu$, $(X_1)_\sharp \prob = \nu$ }},
\end{equation} 
where we denote by $(X_0)_\sharp \prob$ and $(X_1)_\sharp \prob$ the marginal laws of $(X_0, X_1)$.

This variant of the classical Monge-Kantorovich problem \cite{Vill09, AmGiSa08}, with the additional  martingale constraint, originates from intriguing questions of worst case bounds for derivate prices in model-independent finance, see e.g.\ \cite{BeHePe12} for discrete-time and \cite{GaHeTo14} for continuous-time.

In classical optimal transport there is a one-to-one correspondence between discrete-time couplings and continuous-time couplings, in the sense that any continuous-time solution induces a discrete-time solution and any discrete-time solution can be optimally interpolated to a unique continuous-time solution, e.g. by  using  McCann's displacement interpolation \cite{McC95}.
For martingale optimal transport the link between the discrete- and the continuous-time problem is less clear and only \cite{HuSt16, BeLaTo17} provide,  in dimension one, a continuous-time interpretation of the discrete-time transport problem using Skorokhod embedding techniques developed in \cite{BeCoHu16}. However, this approach does not lead to time consistent couplings, i.e.\ such that the induced couplings between two arbitrary intermediate times will be optimal between their marginals.

The aim of this article is to narrow this gap by focusing on a certain class of continuous-time martingale transport problems that naturally appear via a weak length relaxation of the discrete-time problems. 

The main idea to identify the correct class of continuous-time problems and to link it with the discrete-time problems is to transfer to the martingale context the interpretation of the Lagrangian action functional in the Benamou-Brenier formula \cite{BeBr00} as a length functional on the $L^2$ Wasserstein space $\Prob(\R^d)$ of probabilities with finite second moment. More precisely, first one uses the fact that any sufficiently regular curve $(\rho_t)_{t\in [0,1]}$ in $\Prob(\R^d)$, can be represented via the superposition principle \cite[Theorem 8.3.1]{AmGiSa08} as a stochastic process $(X_t)_{t\in [0,1]}$, defined on some probability space, such that $(X_t)_\sharp \prob = \rho_t$ for $t \in [0,1]$. Then for a partition $\pi=\{t_0=0<\ldots<t_n=1\}$ of the unit interval, one considers the discrete energy of $X$ associated with $\pi$
\begin{equation}\label{discreteBBenergy}
 \sum_{i=1}^n (t_{i+1}-t_i) \cdot \E\sqa{\bra{\frac{|X_{t_{i+1}}-X_{t_i}|}{t_{i+1}-t_i}}^2} ~.
\end{equation}
As the mesh of $\pi$ goes to zero, \eqref{discreteBBenergy} converges for sufficiently regular curves $X$  precisely to the Lagrangian action functional (written in probabilistic terms)
$$ \int_0^1 \E [|\dot{ X_t}|^2]~\d t$$
 of the Benamou-Brenier transport formulation, $\dot{X_t}$ being the time derivative of $X$.

Following this idea and taking into account  the scaling properties of martingales implied by the Burkholder-Davis-Gundy inequalities we have the following result (for a rigorous statement  including all the assumptions on the martingale we refer to Theorem \ref{thm:length relax}):

\begin{theorem}\label{thm:relaxation intro}
 Let  $\cc: \R^d \to \R$ be smooth and of bounded growth. Assume  that $X = (X_t)_{t \in [0,1]}$ is a ``sufficiently regular'' martingale with values in $\R^d$, in particular with $\dot{\ang{X}} = \d \ang{X}/ \d t$ well defined for $t \in (0,1)$. Then, the following limit holds:
\begin{align}\label{eq:length relax intro}
 \lim_{\nor{\pi}\to 0} \sum_{t_i\in\pi}  \E\sqa{ \cc\bra{\frac{X_{t_i}-X_{t_{i-1}}}{ \sqrt{t_i- t_{i-1}} }}} (t_i - t_{i-1}) &  = \int_0^1 \E \left[\cc\left(\textstyle{\sqrt{\dot{\qv{X}_t}}} \, Z\right) \right] \d t,
\end{align} 
where $Z$ is a $d$-dimensional  standard normal random variable independent of $X$.
\end{theorem}

Therefore, by interpreting the discrete martingale transport cost as a non-symmetric distance function, the quantity on the right hand side of \eqref{eq:length relax intro} can be seen as the ``length'' of the martingale measured in terms of $\cc$. 
As a direct consequence, there is a natural Benamou-Brenier type formulation of a continuous-time martingale transport problem as the martingale of ``minimal length'' connecting two given measures $\mu$ and $\nu$ which are increasing in convex order,
\begin{equation}
\label{eq:bb-cost-process intro}
\costM(\mu, \nu) := \inf \cur{ \int_0^1 \E\sqa{\cc\bra{\dot{\ang{X}_t} }} \d t },
\end{equation}
where the infimum runs over all martingales connecting $\mu$ and $\nu$ whose quadratic variation process $\qv{X}$ is absolutely continuous w.r.t.\ the Lebesgue measure. Moreover, in \eqref{eq:bb-cost-process intro} we replaced the  initial cost function $\cc$ with its infinitesimal counterpart (still denoted by $\cc$ with a slight abuse of notation), defined on non-negative symmetric $d\times d$ matrices,
\begin{equation}\label{ca-intro}\Sim^d_+ \in a \mapsto   \int_{\R^d} \cc(\sqrt{a}z) \frac{  e^{-\frac{|z|^2}{2}} \d z }{ ( 2 \pi)^{d/2}}.\end{equation}

Notably, this class of cost functions (as well as the larger class of cost functions that additionally depend on time and space considered in the main part of this article) is precisely of the form considered in  \cite{TaTo12}, but the control is restricted on the diffusion term, the drift being null. It could be of interest to provide an argument leading to the general costs considered in \cite{TaTo12} as a relaxation of a semimartingale, perhaps first separating the martingale from the finite variation part via Doob-Meyer decomposition and by scaling differently the two parts.

In the central part of this article, we complement the results of \cite{GaHeTo14, TaTo12} by a new PDE perspective on this problem that  is closer to the original work of Benamou-Brenier \cite{BeBr00} and its extension by Dolbeaut-Nazaret-Savar\'e \cite{DoNaSa09}. Moreover, this point of view usually reduces the complexity of the original problem because we only have to deal with PDEs of the marginals and not with a stochastic process connecting these marginals.

By linking the optimization problem \eqref{eq:bb-cost-process intro} to the classical martingale problem, we show that there is an equivalent formulation in terms of Fokker-Planck equations. Precisely, we let 
\begin{equation}
\label{eq:bb-cost-pde intro}
\costFP(\mu, \nu) := \inf \cur{ \int_0^1 \int_{\R^d} \cc(a_t(x)) \d \varrho_t(x)  \d t \colon \partial_t \varrho = \Tr\bra{ \frac12\nabla^2 a  \varrho},\,  \varrho_0 =\mu, \,   \varrho_1 = \nu },
\end{equation}
where the Fokker-Planck equation $\partial_t \varrho = \Tr\bra{ \frac12\nabla^2 a  \varrho}$ holds in the weak sense (see Section~\ref{sec:notation-basic} for a precise definition). 
In Theorem~\ref{thm:markov-projection} we prove that $\costFP(\mu,\nu)=\costM(\mu,\nu)$ under suitable assumptions on $\cc$. A direct consequence of this formulation as a Lagrangian action minimization problem is that any optimizer will have the time consistency property that all the continuous-time solutions constructed in \cite{HuSt16, BeLaTo17} were lacking. In other words, any optimizer induces a natural interpolation or ``geodesic'' between its marginals.

To find explicit bounds on the cost $\costFP(\mu,\nu)=\costM(\mu,\nu)$ seems however a non-trivial task. Since $\cc$ depends on the diffusion coefficient and not on the quadratic variation of the martingale, the Burkholder-Davis-Gundy inequalities do not imply any bounds on $\costFP(\mu,\nu)=\costM(\mu,\nu)$ in terms of moments, even if the growth of $\cc$ can be controlled. However, assuming that $\cc(a)\lesssim |a|^p$ using Skorokhod embedding techniques we show that in dimension one $\costM(\mu,\nu)<\infty$ if $\nu$ has finite $2p+\eps$-moments for some $\eps>0$. When both $\mu$ and $\nu$ have densities with respect to the Lebesgue measure, one can  also rely on a variant of the Dacorogna-Moser interpolation technique \cite{DaMo90} in this setting.  Notably, this technique, combined with an approximation argument, leads to a constructive PDE proof of Strassen's theorem (Corollary~\ref{cor:strassen}).

Just as in classical transport, the decisive step to understand the optimizer of the martingale Benamou-Brenier problem \eqref{eq:bb-cost-pde intro} is the dual formulation of the variational problem, that we provide in Theorem \ref{thm:duality}. As a byproduct, we also prove the existence of primal optimizers.
An interesting consequence of this duality result is a ``geodesic equation'' of Hamilton-Jacobi type for the optimal potential function $\phi$, which allows us to recognize and construct optimizers. More precisely, given an optimal potential $\phi$, i.e.\ a solution to such a ``geodesic equation'', assuming that all quantities are sufficiently smooth, we define the diffusion coefficient
$$a(t,x)=\nabla \cc^*\bra{ \frac 1 2 \nabla^2 \phi(t,x)},$$
where $\cc^*$ is the Legendre transform of $\cc$. Then,  solving the Fokker-Planck equation $\partial_t\varrho = \Tr\bra{\frac12 \nabla^2 a \varrho}$  with a given initial datum $\varrho_0 = \mu$,  we obtain the solution to the martingale Benamou-Brenier problem between $\mu$ and $\rho_1$. In particular, if we find such a candidate curve connecting the measures $\mu$ and $\nu$ we have found an optimizer for the problem in \eqref{eq:bb-cost-pde intro}.
A particularly nice class of examples is given in dimension one by the cost functions $\cc(a)=a^p$ for $p>1$ (and $a\ge 0$). It then follows that the optimal diffusion coefficient has to solve the pressure equation corresponding to a degenerate porous medium equation (see Remark~\ref{rem:pressure}). Due to the rich literature on porous medium equations this allows us to construct various examples, see Theorem~\ref{thm:existence-pme-hjb} and Example~\ref{ex:friendly giant}. Analogously, in the case $p<1$, solutions should be related to the fast diffusion equation, providing us with another class of examples. In this article, however, we do not pursue this direction. We refer to \cite{BaBeHuKa17} for an investigation of the case $p=1/2$, which allows for an interesting probabilistic representation.

As a final observation, we remark that the diffusive structure of the optimizer to the continuous-time problems, together with our weak length relaxation result, is consistent with the natural interpretation of the discrete-time martingale transport problem as an infinitesimal version of the continuous-time problem. This may also explain why there cannot be a one-to-one correspondence between optimizers of the discrete-time and the continuous-time problems, since mass in the discrete-time problems is known to be split \cite{BeJu16}.  By contrast, in classical transport (at least on $\R^d$) one is not forced to split mass   and one can follow the infinitesimal direction for one unit of time, yielding precisely the connection of the discrete- and continuous-time optimizers.

\paragraph*{Related  literature.}
The one-dimensional discrete-time martingale optimal transport problem is by now well understood due to the seminal work of \cite{BeJu16} for the geometric characterization of optimizers and \cite{BeNuTo16} for a complete duality theory, see also \cite{BeLiOb17}. This was  extended to cover the discrete-time multi-marginal problem in \cite{NuStTa17}. In higher dimensions, a complete picture for the discrete-time problem is still missing. However, there is recent exciting progress in \cite{GhKiLi15, DeTo17, ObSi17}.

The continuous-time version of the martingale optimal transport problem  has been studied in \cite{TaTo12, GaHeTo14, DoSo12, DoSo15} among others, where the main focus of the authors is to establish a duality result which can be interpreted as a robust super/subhedging result. The articles \cite{TaTo12, GaHeTo14} solve the problem by linking it to stochastic control theory, whereas \cite{DoSo12, DoSo15} use a careful discretization procedure of the space variables. Notably, these results  imply numerical schemes to compute the value of the optimisation problem, e.g.\ \cite{BoTa13, TaTo12}.

 As first  observed in \cite{Ot01}, the Benamou-Brenier formula can be interpreted as a  length distance (formally) induced by a Riemannian metric on the space of probability measures with second moment, and it is the basis for the so-called Otto calculus, with applications in PDEs and numerics, see e.g.\ \cite{Vill09, AmGiSa08} for a detailed overview. Moreover, variants of the Benamou-Brenier formulation turned out to be a powerful tool for discrete probability, analysis and geometry \cite{Mi11, Mi13, Ma11, ErMa12}, quantum evolution \cite{CaMa14}, jump diffusions \cite{Er14} and recently to a new approach to the Boltzmann equation \cite{Er16}, and Navier-Stokes equation \cite{ArCrFa17}.

Our duality result could be derived from the results of  \cite{TaTo12} which are established via stochastic control theory. However, we decided to use the PDE point of view to give a short and self-contained proof which we believe is a good example of the complexity reduction arising from such a point of view.

The porous medium equation is a very well studied PDE and we refer to  \cite{Va07} for a comprehensive account. We also quote the recent preprint \cite{BaBeDi17} which considers a degenerate class of porous medium equations in connection with stochastic optimal control problems.

Finally, we quote \cite{BaBeHuKa17} for a nice probabilistic  treatment of the special case $\cc(a)=\Tr(\sqrt{a})$, in connection with stretched Brownian motion.

\paragraph*{Outline.}
 In Section~\ref{sec:notation-basic} we introduce notation and review some technical results useful in the following. In Section~\ref{sec:bb} we introduce and study basic properties of the Benamou-Brenier formulation, in particular the connection with the PDE formulation (Theorem~\ref{thm:markov-projection}).  In Section~\ref{sec:duality} the duality result is established, together with examples, while Section~\ref{sec:one-dim} focuses on the one-dimensional case, providing sufficient conditions for finiteness of the transportation cost and links with porous medium equations. Conclusions and open problems are stated in Section~\ref{sec:conclusion}. Appendix~\ref{app:proof} contains the proof of the rigorous version of Theorem~\ref{thm:relaxation intro}.

\paragraph*{Acknowledgements.} Both authors thank M. Beiglb\"ock, J.\ Backhoff Veraguas and S.\ K\"allblad for interesting discussions on martingale optimal transportation.  The first author acknowledges support by the Deutsche Forschungsgemeinschaft through the CRC 1060 \emph{The Mathematics of emergent effects}, the \emph{Hausdorff Center for Mathematics} and by the  FWF-grant Y00782.  The second author is a member of the Gnampa group (INdAM).

\section{Notation and basic facts}\label{sec:notation-basic}

For $d \ge 1$, let $\Sim^d \subseteq \R^{d \times d}$ denote the space of symmetric matrices $a = a^\tau$ and write $a \in \Sim^d_+$ or $a \ge 0$ if $a \in \Sim^d$ is non-negative definite. We endow $\Sim^d$ with the Hilbert-Schmidt norm $|a| :=  \sqrt{\Tr\bra{a^2}}$, $\Tr$ denoting the trace operator. 

Let $C_b([0,1]\times \R^d)$,  $C_b^{1,2}([0,1]\times \R^d)$ be respectively the spaces of continuous functions $\varphi(t,x) = \varphi_t(x)$ on $[0,1]\times \R^d$ and of functions differentiable once with respect to $t$ (and write $\partial_t \varphi$) and twice  with respect to $x$ ($\nabla \varphi$, $\nabla^2 \varphi$) in $(0,1) \times \R^d$, with $\partial_t \varphi, \nabla^2 \varphi$ uniformly bounded. For $r \in [0, +\infty)$, let $B_r$, $\overline{B_r} \subseteq \R^d$ denote respectively the open and closed balls centred at $0$ in $\R^d$ and write $C_b^{1,2}([0,1]\times \overline{B_r})$ for the space of functions differentiable once with respect to $t$  and twice  with respect to $x$  in $(0,1)\times B_r$ with uniformly continuous derivatives, so that they extend to functions in $C^{1,2}_b([0,1]\times \R^d$).

Given $E \subseteq \R^k$ Borel we write $\Prob(E)$ for the set of probability measures on $E$ with finite second moment, endowed with the narrow topology (i.e.\ in duality with bounded continuous functions) together with convergence of second moments). We write $\meas(E; \Sim^d_+)$ for the space of measures $\mu$ on $E$ with values in the vector space $\Sim^d$ such that $\mu(A) \in \Sim^d_+$ for every Borel $A \subset E$, or, equivalently, such that the matrix $\sigma$ in the polar decomposition of $\mu = \abs{\mu} \sigma$ belongs $|\mu|$-a.e.\ to $\Sim^d_+$.  We write $\chi_A$ for the indicator function of a set $A$.

We say that $\mu$, $\nu \in \Prob(\R^d)$ are in convex order if for every convex $\varphi:\R^d \to \R$ one has $\int_{\R^d} \varphi \d \mu  \le \int_{\R^d} \varphi \d \nu$.

\paragraph*{Cost functionals.}
Given $\cc: \Sim^d_+ \to \R \cup \cur{+\infty}$, $a \mapsto \cc(a)$ (or equivalently $\cc: \Sim^d \to \R \cup \cur{+\infty}$ with $\cc(a) = +\infty$ if $a \notin \Sim^d_+$), its Legendre transform $\cc^*:\Sim^d \to  \R \cup \cur{+\infty}$ is defined for $u \in \Sim^d$ as
\[ \cc^*(u) := \sup_{a \in   \Sim^d_+}\cur{ \Tr\bra{au} - \cc(a)}.\]
The map $\cc^*$ is then convex and lower semicontinuous (l.s.c.).  If $\cc$ is strictly convex, then  by  \cite[Theorem 26.3]{rockafellar_convex_1970} $\cc^*$ is continuously differentiable with
\begin{equation}\label{eq:optimality-c-c-star} \cc^*(u) = \Tr\bra{( \nabla \cc^*(u) ) u}- \cc( \nabla \cc^*(u) ) \quad \text{for $u \in \Sim^d$.}\end{equation}
In particular, $\nabla \cc^*(u) \in \Sim_+^d$. 
For $p \in (1,+\infty)$, we say that $\cc: \Sim_d^+ \to \R$ is $p$-{coercive} if there exists $\lambda>0$ such that $\cc(a) \ge \lambda |a|^p$, for $a \in \Sim_d^+$ and that it has $p$-growth if $\cc(a) \le \lambda |a|^p$, for $a \in \Sim_d^+$. The Legendre transform of a $p$-coercive function has $q$-growth and that of a function with $p$-growth is $q$-coercive, with $q = p/(p-1)$. If $\cc$ is strictly convex, from \eqref{eq:optimality-c-c-star} we obtain that if $\cc$ is $p$-coercive then $|\nabla \cc^*|$ has $(q-1)$-growth. Finally, we say that $\cc$ is $p$-admissible if it is strictly convex, $p$-coercive and has $p$-growth.

In this paper, we consider Borel cost functionals  $\cc: (0,1)\times \R^d \times \Sim^d_+ \to \R\cup \cur{+\infty}$, $\cc(t,x,a)$ and their partial Legendre transform with respect to the variable $a \in \Sim^d_+$: all the definitions given above, i.e.\ (strict) convexity, $p$-coercivity, $p$-growth, and $p$-admissability are to be considered with respect to the variable $a$ (or the dual variable $u \in \Sim^d$) and must hold with uniform constants with respect to $(t,x) \in  (0,1)\times \R^d$.

\paragraph*{Martingales.}
Throughout this paper, we consider only stochastic processes with continuous trajectories and always assume the validity of the ``usual conditions'' on filtered probability spaces $(\Omega, \cA, \prob, (\cF)_{t \in [0,1]})$. We also never consider a fixed probability space, but rather allow our optimization to range over all spaces.

A (continuous) real-valued stochastic process $M=(M_t)_{t \in [0,1]}$ defined on a filtered probability space $(\Omega, \cA, \prob, (\cF)_{t \in [0,1]})$ is a martingale if it is adapted, i.e.\ for every $t \in [0,1]$, $M_t$ is $\cF_t$-measurable, $\E[\abs{M_t}]<\infty$, and, for $s<t \in [0,1]$, $\E[M_t | \cF_s] = M_s$. Throughout this paper we consider only square integrable martingales $M$, i.e.\ such that $\E[ |M_1|^2]<\infty$. 

A martingale $M$ has finite quadratic variation \cite[Chapter I, \textsection 2]{ReYo99} if there exists a non-negative adapted process $\ang{M} = (\ang{M}_t)_{t\in [0,1]}$ such that, for any sequence of partitions $(\pi^n)_{n \ge 1}$ of $[0,1]$, whose diameter $\abs{\pi^n} \to 0$ as $n \to +\infty$, then the limit in probability
\[ \lim_{n \to +\infty} \sum_{t_i \in [0,t] \cap \pi^n} \bra{M_{t_i} - M_{t_{i-1}}}^2 = \ang{M}_t\]
holds. By \cite[Theorem~1.3]{ReYo99} (if $M$ is square integrable), $\ang{M}$ exists and is the unique continuous, increasing, adapted process such that $\ang{M}_0 = 0$ and $(M^2_t - \ang{M}_t)_{ t \in [0,1]}$ is a martingale (not necessarily square integrable).

Throughout this paper, we write that  a martingale $M$ has absolutely continuous quadratic variation if for some (progressively measurable) process  $\dot{\ang{M}} =(\dot{\ang{M}}_t)_{t \in [0,1]}$, one has
\[ \ang{M}_t = \int_0^t \dot{\ang{M}}_s \d s, \quad \text{for $t \in [0,1]$,}\]
and write $M \in \ACQV{p}$ if $\E\sqa{\int_0^1 |\dot{\ang{M}}_t|^p \d t}<\infty$. Similar definitions and properties hold for martingales taking values in $\R^d$, arguing componentwise: in particular, the processes $\ang{M}$, $\dot{\ang{M}}$ take values in $\Sim^d_+$.

\paragraph*{Fokker-Planck equations.}

Given $(\varrho_t)_{t \in [0,1]} \subseteq \Prob(\R^d)$ continuous and a Borel map $a: (0,1) \times \R^d \to \Sim^d_+$, we say that the Fokker-Planck equation
\begin{equation}\label{eq:fpe} \partial_t \varrho = \Tr\bra{ \frac12\nabla^2 a \varrho}, \quad \text{in $(0,1) \times \R^d$,} \tag{$\mathsf{FPE}$}\end{equation}
holds if 
 $\int_0^1 \int_{\R^d} \abs{a_t}\d \varrho_t < \infty$ and for $\varphi \in C^{1,2}_b( [0,1]\times \R^d)$ one has
\begin{equation}\label{eq:fpe-test} \int_0^1 \int_{\R^d} \bra{ \partial_t \varphi  +  \Tr\bra{\frac12 a_t \nabla^2 \varphi }} \d \varrho_t \d t = \int_{\R^d} \varphi_1 \d \varrho_1 - \int_{\R^d} \varphi_0 \d \varrho_0.\end{equation}

It is technically useful to extend the notion of Fokker-Planck equation to general measures, essentially defining $\ca := a \varrho$, so that the equation becomes linear. Precisely, given $\varrho \in \Prob( (0,1)\times \R^d )$, $\ca \in \meas( (0,1)\times \R^d; \Sim^d_+)$, we say (with a slight abuse of notation) that $\partial_t \varrho = \Tr\bra{\frac12 \nabla^2 \ca}$ holds in $(0,1) \times \R^d$ if one can disintegrate $\varrho = \int_0^1 \varrho_t \d t$, $\ca = \int_0^1 \ca_t \d t$, and the identity $\int_{(0,1)\times \R^d} \bra{ \partial \varphi \d \varrho +  \Tr\bra{\frac12 \nabla^2 \varphi \d \ca}} = 0$ holds for every $\varphi \in C^{1,2}_b( [0,1]\times \R^d)$ compactly supported in  $(0,1) \times \R^d$.  Notice that, when $\ca$ is absolutely continuous with respect to $\varrho$, then the two given notions coincide, letting $a:= \frac{\d \ca}{\d \varrho}$ be the Radon-Nikodym derivative, up to providing a continuous representative for $(\varrho_t)_{t \in [0,1]}$, which can be always done, arguing analogously as in \cite[Lemma 8.1.2]{AmGiSa08}, see also \cite[Remark 2.3]{trevisan_well-posedness_2016}.

\paragraph*{Martingale problems.} We say that a continuous stochastic process $(X_t)_{t \in [0,1]}$, taking values in $\R^d$, and defined on some filtered probability space  $(\Omega, \cA, \prob, (\cF)_{t \in [0,1]})$ is a solution to the martingale problem \cite[Chapter 6]{stroock_multidimensional_2006} associated to a Borel function  $a: (0,1) \times \R^d \to \Sim^d_+$ (diffusion coefficient) if 
\[  \E\sqa{ \int_0^1 \abs{a_t(X_t)} \d t} < \infty\]
and for every $\varphi \in C^{1,2}_b([0,1]\times \R^d)$ the process
\begin{equation}\label{eq:martingale} t \mapsto \varphi(t, X_t) - \int_0^t \bra{ \partial_t \varphi(s, X_s) + \Tr\bra{ \frac12 a_s(X_s) \nabla^2 \varphi(s, X_s) }} \d s\end{equation}
is a martingale. An application of It\^o's formula shows that its quadratic variation  is 
\[ t \mapsto \int_0^t \bra{ a_s(X_s) \nabla \varphi_s(X_s)} \cdot  \nabla \varphi_s(X_s)  \d s,\]
hence the martingale is in $\ACQV{1}$. Letting $\varphi(x) = x$, then $X$ itself is a martingale, with density of quadratic variation $\dot{\ang{  X}}_t = a_t(X_t)$.

Since the expectation of the martingale \eqref{eq:martingale} is constant, we see that the $1$-marginal laws of $X$, i.e. the (continuous) curve $t\mapsto \varrho_t := (X_t)_\sharp \P \in \Prob(\R^d)$ solves \eqref{eq:fpe}. Moreover, the curve $(\varrho_t)_{t \in [0,1]}$ is increasing with respect to the convex order.

A converse result holds (see \cite[Theorem 2.6]{figalli_existence_2008} for a proof in case of bounded $a$, \cite[Theorem 2.5]{trevisan_well-posedness_2016} for the general case). 

\begin{theorem}\label{thm:sp}
Let  $(\varrho_t)_{t \in [0,1]} \subseteq \Prob(\R^d)$ be continuous and $a: (0,1) \times \R^d \to  \Sim_d^+$ be Borel solving \eqref{eq:fpe}. Then, there exists a continuous process  $(X_t)_{t \in [0,1]}$ defined on some filtered probability space $(\Omega, \cA, \P, (\cF_t)_{t \in [0,1]})$, solving the martingale problem associated to $a$, such that $\varrho_t = (X_t)_\sharp \P$, for every $t \in [0,1]$.
\end{theorem}

\section{A Benamou-Brenier type problem}\label{sec:bb}

In this section we introduce and study basic properties of  Benamou-Brenier type problems which naturally appear as length-type functionals with respect to the discrete-time martingale transport cost. Indeed, by introducing suitably normalized ``cumulative'' transport costs, associated to a partition $\pi = \cur{t_0 =0 < \ldots < t_n=1} \subseteq [0,1]$, and investigating their limit as $\nor{\pi} \to 0$, where $\nor{\pi} = \sup_{i=1, \ldots,n} | t_i-t_{i-1}|$, we can show the following result (its proof is postponed to the Appendix~\ref{app:proof}).

\begin{theorem}\label{thm:length relax}
 Let $p \ge 1$, $\cc \in C(\R^d)$ satisfy $|\cc(y)-\cc(x)| \le \lambda (1+|x|^{2p-1} + |y|^{2p-1})|y-x|$ for $x, y \in \R^d$, for some $\lambda\ge 0$. Let $X\in\ACQV{p}$ with $t\mapsto\dot{\qv{X}}_t$ $\prob$-a.s.\ continuous and let $Z$ be a $d$-dimensional  standard normal random variable independent of $X$. Then the following limits hold:
\begin{align}\label{eq:length relax}
 \lim_{\nor{\pi}\to 0} \sum_{t_i\in\pi}  \E\sqa{ \cc\bra{\frac{X_{t_i}-X_{t_{i-1}}}{ \sqrt{t_i- t_{i-1}} }}} (t_i - t_{i-1}) &  = \int_0^1 \E \left[\cc\left(\textstyle{\sqrt{\dot{\qv{X}}}} \, Z\right) \right] ds,\\\label{eq:length relax two}
 \lim_{\nor{\pi}\to 0} \sum_{t_i\in\pi} \abs{ \E\sqa{ \cc\bra{\frac{X_{t_i}-X_{t_{i-1}}}{ \sqrt{t_i- t_{i-1}} }}}}^{1/p} (t_i - t_{i-1})  & = \int_0^1 \abs{\E\left[\cc\left(\textstyle{\sqrt{\dot{\qv{X}}}} \, Z\right)\right]}^{1/p} ds~.
\end{align} 
\end{theorem}

Clearly, one could consider other  ``cumulative'' transport costs, but a common feature should be that, if $\cc(x,y)$ depends on the difference $y-x$, then letting $x = X_{t}$, $y = X_{s}$ for a continuous martingale $(X_t)_{t \in [0,1]}$, a rescaling factor $\sqrt{t-s}$ should appear. Notice also that, if $\cc$ is an odd function, the right hand side in \eqref{eq:length relax} is identically zero, by independence of $Z$ and $X$.

In view of Theorem~\ref{thm:length relax}, we introduce the following minimization problem, as a continuous-time martingale optimal transport problem. For a cost functional $\cc: [0,1]\times \R^d \times \Sim^d_+ \to \R \cup \cur{+\infty}$ and $\mu$, $\nu \in \Prob(\R^d)$ we set
\begin{equation}
\label{eq:bb-cost-process}
\costM(\mu, \nu) := \inf \cur{ \int_0^1 \esp{ \cc( t, X_t, \dot{\ang{ X}}_t)} \d t \colon  X \in \ACQV{},  (X_0)_\sharp \P =  \mu, \,   (X_1)_\sharp \P = \nu },
\end{equation}
which can be interpreted as the infimum  over the length of martingale curves connecting $\mu$ to $\nu$, where the length is the integral of the ``speed'' $\cc(t, X_t,\dot{\ang{ X}}_t)$. Notice that also the filtered probability space $(\Omega, \cA, \P, (\cF_t)_{t \in [0,1]})$ where $(X_t)_{t \in [0,1]}$ is defined is allowed to vary in the formulation above. This problem is a ``martingale'' analogue of the Benamou-Brenier dynamical formulation of optimal transport \cite{BeBr99}.  

The choice of looking at martingales defined on the time interval $[0,1]$ is arbitrary in \eqref{eq:bb-cost-process}.

%
%
%

\begin{lemma}[time-changes and unit speed geodesics]\label{rem:time-rescaling}
 Assume that the cost function $\cc(t,x, \cdot)$ is $p$-homogeneous for some $p >1$, i.e. $\cc(t,x, \lambda a) = \lambda^p \cc(t,x,a)$ for $\lambda \ge 0$. Then, for any $T>0$ there holds
\[ \bra{ \costM(\mu, \nu) }^{1/p}= \inf \cur{ \int_	0^T \abs{ \esp{ \cc( t, X_t, \dot{\ang{ X}}_t)} }^{1/p} \d t \colon  X \in \ACQV{},  (X_0)_\sharp \P =  \mu, \,   (X_T)_\sharp \P = \nu }. \]
Moreover, letting $\varrho_t:=(X_t)_\sharp\prob$, one has $\bra{\costM(\varrho_s, \varrho_t)}^{1/p} = \abs{t-s} \bra{ \costM(\varrho_0, \varrho_1)}^{1/p}$, i.e.\ optimizers of \eqref{eq:bb-cost-process} are unit speed geodesics.
\end{lemma}

The key observation for the first part of Lemma \ref{rem:time-rescaling} is that for a deterministic time change $\alpha$, the quadratic variation of $(X_{\alpha(t)})_t$ satisfies $\dot {\ang{X_\alpha}}_t = \dot {\alpha}_t \dot {\ang{X}}_{\alpha(t)}$ so that the cost on different intervals changes in a controlled way. Choosing $\alpha$ as the inverse of $t \mapsto \int_0^t  \esp{ \cc( r, X_r, \dot{\ang{ X}}_r)} \d r$ yields (after some computations) the result. Combining this with time rescaling yields the second assertion. We refer to \cite[Theorem 5.4]{DoNaSa09} for a complete derivation in the deterministic setting. One only needs to adapt the scaling for the quadratic variation in the time change argument.

To address well-posedness of \eqref{eq:bb-cost-process}, we introduce a second minimization problem, defined over curves of probability measures solving Fokker-Planck equations, for which we show equivalence with \eqref{eq:bb-cost-process} for $p$-admissible  costs.  
Indeed, if $X = (X_t)_{t \in [0,1]}$ is a solution to the martingale problem associated to some diffusion coefficient $a$, then $\dot{\ang{X}}_t =  a_t(X_t)$ and we rewrite
\[ \int_0^1 \E[ \cc(t,X_t, \dot { \ang{ X}}_t)]  \d t =\int_0^1  \int_{\R^d} \cc(t, x,  a_t(x) ) \d \varrho_t(x) \d t ,\] where $\varrho_t = (X_t)_{\sharp} \prob$ is the marginal law of $X$ at time $t$. It follows that the functional in \eqref{eq:bb-cost-process} actually depends on $(\varrho_t)_{t \in [0,1]}$ and $a$ only, which are related by \eqref{eq:fpe}. We are led by this consideration to introduce the following minimization problem:
\begin{equation}
\label{eq:bb-cost-pde}
\costFP(\mu, \nu) := \inf \cur{ \int_0^1 \int_{\R^d} \cc(t, x, a_t(x)) \d \varrho_t(x)  \d t \colon \partial_t \varrho = \Tr\bra{\frac12 \nabla^2 a \varrho},\,  \varrho_0 =\mu, \,   \varrho_1 = \nu }.
\end{equation}

Inequality $\costM(\mu, \nu) \le \costFP(\mu, \nu)$ always holds. Indeed, by Theorem~\ref{thm:sp}, any solution $(\varrho_t)_{t \in [0,1]}$ to \eqref{eq:fpe} can be lifted to a solution $(X_t)_{t \in [0,1]}$ to the martingale problem associated to $a$, on some filtered probability space $(\Omega, \cA, (\cF_t)_{t \in [0,1]}, \prob)$, and the consideration above, which led to the introduction of \eqref{eq:bb-cost-pde}, applies. To prove the converse inequality, hence equality $\costFP = \costM$, we essentially argue that, given any martingale $(X_t)_{t \in [0,1]}$ one can always find a solution to some martingale problem (possibly on a larger space) with the same marginals and smaller transport cost, provided that $\cc$ is $p$-admissible. The  argument is a variant of  \cite[Theorem~8.3.1]{AmGiSa08} in the martingale setting.

\begin{theorem}\label{thm:markov-projection}
For $p \in (1,+\infty)$ let $c: [0,1]\times \R^d \times \Sim_d^+ \to \R \cup \cur{+\infty}$ be  $p$-admissible, let $X \in \ACQV{p}$ and set $\varrho_t = (X_t)_\sharp \P$, for $t \in [0,1]$. Then, there exists $a: [0,1]\times \R^d \to \Sim_d^+$ such that $(\varrho, a)$ satisfy \eqref{eq:fpe}  and
\begin{equation}\label{eq:projection-martingale-mp}  \int_0^1  \int_{\R^d} c(t, x, a_t(x)) \d \varrho_t(x) \le \int_0^1 \esp{c( t, X_t, \dot{\ang{ X}}_t)} \d t .\end{equation}
In particular, the identity  $\costM(\mu, \nu) = \costFP(\mu, \nu)$ holds for every $\mu$, $\nu \in \Prob(\R^d)$.
\end{theorem}

\begin{proof}
Let $\mathcal{E}$ denote the right hand side in \eqref{eq:projection-martingale-mp}, which is finite by the assumption $X \in \ACQV{p}$. We introduce the linear functional $\cL$, on $C^{1,2}_b ([0,1] \times \R^d)$,
\begin{equation}\label{eq:def-l-projection-martingale}\cL \varphi  := \int_{\R^d} \varphi(1, x) \d \varrho_1 - \int_{\R^d} \varphi(0, x) \d \varrho_0(x) - \int_0^1 \int_{\R^d} \partial_t \varphi (t, x) \d \varrho_t (x) \d t,\end{equation}
and we notice that
\[ \begin{split} \cL \varphi  & =  \esp{ \varphi(1, X_1) - \varphi(0, X_0) -  \int_0^1  \partial_t \varphi (t, X_t)  \d t} \quad \text{since $\varrho_t = (X_t)_\sharp \P$,} \\
& =\esp{  \int_0^1  \nabla \varphi (t, X_t)  \d X_t +  \int_0^1  \Tr\bra{\frac12 \dot { \ang{ X}}_t  \nabla^2 \varphi (t, X_t)} \d t} \quad \text{by It\^o's formula,}\\
& = \esp{\int_0^1    \Tr\bra{\frac12 \dot{\ang{ X}}_t  \nabla^2 \varphi (t, X_t) }\d t} \quad \text{for the stochastic integral is a martingale.}\\
\end{split}\]
By H\"older's inequality, we deduce with $q=\frac{p}{p-1}$,
\[ \begin{split} \abs{ \cL \varphi} &  \le \frac12 \esp{\int_0^1  \abs{ \dot { \ang{ X}}_t } \abs{ \nabla^2 \varphi (t, X_t)} \d t} \le  \esp{\int_0^1  \abs{ \dot { \ang{ X}}_t }^p \d t}^{1/p}  \esp{\int_0^1  \abs{ \nabla^2 \varphi(t, X_t) }^q \d t}^{1/q}\\
& = \norm{  \dot {\ang{X}}}_{L^p} \bra{\int_0^1 \int_{\R^d} \abs{\nabla^2 \varphi(t, x) }^q \d \varrho_t \d t}^{1/q} = \norm{ \dot {\ang{X}}}_{L^p} \norm{ \nabla^2 \varphi }_{L^q(\varrho)}.\end{split}\]
where $L^q(\varrho)$ denote the Lebesgue spaces with respect to $\int_0^1\varrho_t \d  t$.
As a consequence, the linear functional $\tilde{\cL}(\nabla^2 \varphi) := \cL\varphi$ is actually well-defined and continuous  on the space
\[ \cV := \cur{ \nabla^2 \varphi \colon \varphi \in C^{1,2}_b ([0,1] \times \R^d) } \subseteq L^q( \varrho; \Sim^d),\]
and extends by continuity to the closure $\overline{\cV}\subseteq L^q( \varrho; \Sim^d)$. The inequality
\[ \begin{split}  \cL \varphi & = \esp{\int_0^1\Tr\bra{   \dot { \ang{ X}}_t \frac12 \nabla^2 \varphi (t, X_t)} \d t}  \\
& \le \esp{\int_0^1   c\bra{t, X_t, \dot { \ang{ X}}_t}  +  \cc^*\bra{t, X_t, \frac12 \nabla^2 \varphi (t, X_t)} \d t}\\
& = \cE +  \int_{(0,1)\times \R^d}\cc^*\bra{y, \frac12\nabla^2 \varphi(y)} \d \varrho(y)
\end{split}\]
by continuity (for $\cc^*$ has $q$-growth) extends to
\begin{equation}\label{eq:upper-bound-projection-proof}\int_{[0,1]\times \R^d} \cc^*\bra{y,\frac12 u(y)} \d {\varrho}(y) - \tilde{\cL}u \ge - \cE, \quad \text{for every $u \in \overline{\cV}$.}\end{equation}
Moreover, since $\cc^*$ is $q$-coercive, the functional 
\[ 
 \int_{[0,1]\times \R^d}\cc^*\bra{y,  \frac12 u(y)} \d \varrho- \tilde{\cL} u \]
is convex and coercive on $\overline{\cV}$, hence it attains its minimum at some $\bar{u} \in \overline{\cV}$. Using the fact that $\nabla_u \cc^*$ exists, is continuous and has $(q-1)$-growth, the optimality condition reads
\[ \int_{[0,1]\times \R^d} \Tr\bra{ \frac12 \bra{ \nabla_{u}\cc^*\bra{y, \frac12\bar{u}(y)}}  u (y) } \d{\varrho}(y)= \tilde{\cL} u \quad \text{ for every  $u \in \overline{\cV}$.}\]
Letting $u  = \nabla^2 \varphi$ for $\varphi \in C^{1,2}_b ([0,1] \times \R^d)$, recalling that $\tilde \cL u =\cL \varphi$ and \eqref{eq:def-l-projection-martingale}, we deduce that \eqref{eq:fpe} holds with 
\[  a_t(x) := \nabla_{u}\cc^*\bra{t,x, \frac12\bar{u}(t,x)}. \]
Finally, choosing $u = \bar{u} \in \overline{\cV}$ and using \eqref{eq:upper-bound-projection-proof}, we deduce
\[  \int_{[0,1]\times \R^d}\Tr\bra{ \bra{ \nabla_{u}\cc^*\bra{y, \frac12\bar{u}(y)}} \frac12\bar u (y)}  \d{\varrho}(y) \le \cE + \int_{[0,1]\times \R^d}\cc^*\bra{y, \frac12\bar{u} (y)} \d{\varrho}(y),\]
hence, by \eqref{eq:optimality-c-c-star}, we conclude that
\begin{align*}
& \int_{[0,1]\times \R^d} c\bra{y, a(y) } \d  \varrho(y)\\
 =&  \int_{[0,1]\times \R^d} \Tr\bra{ \bra{\nabla_{u}\cc^*\bra{y, \frac12\bar{u}(y)} } \frac12\bar u (y) } -\cc^*\bra{y, \frac12\bar{u} (y)}  \d{\varrho}(y)  \le \cE. \ \qedhere
\end{align*} 
\end{proof}

\begin{remark}
To show that the  inequality \eqref{eq:projection-martingale-mp} may be strict, e.g.\ consider  the case $d=1$, $c(t,x,a) = a^2$ and let $(\Omega, \cA, (\cF_t)_{t \in [0,1]}, \P)$ be a filtered probability space where a real valued Brownian motion $(B_t)_{t \in [0,1]}$ and  a  $\cF_0$-measurable, uniform random variable $Z$ with values on $\cur{1,2}$ are defined. The martingale $X_t = Z B_t$ belongs to  $\ACQV{2}$ with $\dot{\ang{ X}}_t = Z^2$, and 
\[\varrho_t (\d x) = \frac 1 2\bra{ \varrho^1_t(x) +  \varrho^2_t(x) }\d x,\]
with $\varrho^1_t$ and $\varrho^2_t$ centered Gaussian densities of variances respectively $t$ and $4t$. The proof of Theorem~\ref{thm:markov-projection} gives then
\[ a(t,x) = \frac{\varrho^1_t (x) + 4 \varrho^2_t(x)}{\varrho^1_t (x) + \varrho^2_t(x)} = s + 4 (1-s),\]
with $s(t,x) = \varrho^1_t (x) / \varrho^1_t (x) + \varrho^2_t(x) \in (0,1)$, so that  Jensen's inequality gives $|a(t,x)|^2 < s + 16 (1-s)$. By integration with respect to $\varrho_t(x) \d x \d t$ we conclude that the inequality \eqref{eq:projection-martingale-mp} is strict.
\end{remark}

We discuss here two straightforward properties of the costs \eqref{eq:bb-cost-process} and \eqref{eq:bb-cost-pde}, namely localization and behaviour with respect to convolution.

To localize, e.g.\ on a ball $B_r$ of radius $r>0$, given a martingale $X$ we introduce the stopping time (exit time)
\[\tau_r(X) :=\inf\cur{t\geq 0 \, : \, X_t \notin B_r}\]
We introduce then the ``discounted'' cost 
\[ 
\costM^r(\mu, \nu) := \inf \cur{ \esp{  \int_0^{\tau_r(X)} \cc( t, X_t, \dot{\ang{ X}}_t) \d t } \colon  X \in \ACQV{},  (X_0)_\sharp \prob =  \mu, \,   (X_1)_\sharp \prob = \nu }.
\]
for which the following result holds (but see also Remark \ref{rem:local}).

\begin{lemma}[localization]
Let $\cc:(0,1)\times \R^d\times \Sim_+^d\to [0,\infty]$ satisfy $\cc(t,x,0)=0$ and let $\mu$, $\nu \in \Prob(\R^d)$. Then
$$\lim_{r\to\infty} \costM^r(\mu, \nu)  = \sup_{r>0} \costM^r(\mu, \nu) = \costM(\mu,\nu),$$
\end{lemma}

\begin{proof}
If $X \in \ACQV{}$, then $\dot{\ang{X^{r}_t}} = \dot{\ang{X}} \chi_{t \le \tau_r(X)}$ hence the sequence $\costM^r(\mu, \nu)$ is increasing and bounded by $\costM(\mu,\nu)$, and the limit follows from monotone convergence. 
\end{proof}

Given a measure $\sigma \in \Prob(\R^d)$, one can easily prove that, if $\cc(t,x,a) = \cc(t, a)$ does not depend on $x \in \R^d$, then the convolution operation with $\sigma$ is a contraction of the cost $\costM$, i.e.\
\begin{equation}\label{eq:convolution-1} \costM(\mu \conv \sigma, \nu \conv \sigma) \le  \costM(\mu, \nu) \quad \text{for $\mu$, $\nu \in \Prob(\R^d)$.}\end{equation}
Indeed, given any martingale $X \in \ACQV{}$ with $(X_0)_\sharp \prob =  \mu$, $(X_1)_\sharp \prob = \nu$, we may enlarge the filtration so that there exists a random variable $Y$ independent of $X$ and $\cF_0$-measurable, with law $\sigma$. Then, the process $Z_t := X_t + Y$ is a martingale, with $\ang{Z} = \ang{X}$ and with law at $0$ (respectively, at $1$) given by $\mu \conv \sigma$ (respectively, $\nu \conv \sigma$). Hence,
\[ \int_0^{1}  \esp{ \cc( t, \dot{\ang{ Z}}_t)}  \d t  = \int_0^{1}  \esp{ \cc( t, \dot{\ang{ X}}_t)}  \d t ,\]
and the inequality \eqref{eq:convolution-1} follows. Notice also that as $\sigma \to \delta_0$ we obtain that the costs converge. The fact that $\cc$ is independent of $x$ can be relaxed in a concavity assumption for $x \mapsto \cc(t,x,a)$. A similar argument, in the formulation \eqref{eq:bb-cost-pde}, gives the next result. 

\begin{lemma}[convolution]
Let $\sigma \in C^2(\R^d)$ be a probability density, positive everywhere and  with $|\nabla^i \sigma| \le \lambda \sigma$ for $i \in \cur{1,2}$ (for some $\lambda>0$). Then, if $(\varrho_t)_{t \in [0,1]}$, $a: (0,1) \times \R^d \to \Sim^d_+$ solve \eqref{eq:fpe}, then $\tilde{\varrho}_t := \varrho_t \conv \sigma$ and 
\[ \tilde{a} := \frac{ \d \bra{  \varrho a \conv \sigma} } {\d (\varrho \conv \sigma)} \in C^2(\R^d; \Sim^d_+) \]
solve $\partial_t \tilde{\varrho}  = \Tr\bra{\frac12\nabla^2\tilde{a} \tilde{\varrho}}$ in $(0,1) \times \R^d$. Moreover, if $\cc = \cc(t,a)$ does not depend on $x \in \R^d$ and is such that $a \mapsto \cc(t, a)$ is convex, then
\[ \int_0^1 \int_{\R^d} \cc(t, \tilde{a}(t,x)) \d \tilde{ \varrho}_t (x) \d t  \le \int_0^1 \int_{\R^d} \cc(t, a(t,x)) \d  \varrho_t (x) \d t .\]
\end{lemma}

For a proof, we refer to \cite[Lemma A.1]{trevisan_well-posedness_2016} and \cite[Lemma 8.1.10]{AmGiSa08}. The advantage with respect to \eqref{eq:convolution-1} is that the diffusion coefficient  becomes smooth, if $\sigma$ is chosen appropriately.

\section{Duality}\label{sec:duality}
In this section, we introduce a dual problem to \eqref{eq:bb-cost-pde} which allows us to give optimality conditions for the primal problem. A key role is played by the following (backward) Hamilton-Jacobi-Bellman PDE, 
\begin{equation}
\label{eq:pde 2} \partial_t \varphi(t,x) = -\cc^*\bra{t,x, \frac12 \nabla^2 \varphi(t,x)}. \tag{$\mathsf{HJB}$}
\end{equation}
 We choose to consider only classical solutions and avoid the use of viscosity solutions, although they are a standard tool in such optimal control problems \cite{fleming_controlled_2006}, since their use does not seem to provide further insights to our problem, except for Remark~\ref{rem:viscosity}. We  work on the domains $[0,1]\times \overline{B_r}$ ($r > 0$) and $[0,1]\times \R^d$, specifying boundary conditions in the former situation. 

\begin{definition}[solutions to \eqref{eq:pde 2}]
Let $\Omega = B_r$ or $\Omega = \R^d$. We say that $\varphi \in C^{1,2}_b( [0,1]\times \overline{\Omega})$ is a solution to \eqref{eq:pde 2} if identity holds in \eqref{eq:pde 2} for every $(t,x) \in (0,1) \times \Omega$. We say that $\varphi \in C^{1,2}_b( [0,1]\times \overline{B_r})$ is a super-solution (respectively, sub-solution) to \eqref{eq:pde 2} if inequality $\le$ (respectively, $\ge$) holds, instead of equality, at every $(t,x) \in (0,1) \times \Omega$.
When $\Omega = B_r$, we say that the boundary condition $\nabla^2 \varphi = 0$ holds if (the continuous extension of $\nabla^2 \varphi$ satisfies) $\nabla^2 \varphi(t,x) = 0$, for $(t,x) \in [0,1] \times \partial B_r$.
\end{definition}

\begin{remark}[comparison principle]\label{rem:comparison}
The terms super-solution and sub-solution are justified by the validity of a (standard) comparison result.  For $r >0$, let $\phi, \psi  \in C^{1,2}_b([0,1]\times \overline{B_r})$ be respectively a sub-solution and a super-solution  to \eqref{eq:pde 2}, with $\phi(1, x) \le \psi(1, x)$ for every  $x \in B_r$,  and boundary condition $\nabla^2 \varphi = 0$.  
Then, $\phi \le \psi$ on $[0,1]\times \overline{B_r}$.
\end{remark}

As a first observation, we notice that, if $\varphi \in C^{1,2}_b([0,1]\times \R^d)$ is a super-solution to \eqref{eq:pde 2}, then given any solution $(\varrho_t)_{t \in [0,1]} \subseteq \Prob(\R^d)$ to \eqref{eq:fpe} for some $a:(0,1)\times \R^d \to  \Sim^d_+$, then
\begin{equation}\label{eq:duality-inequality}\begin{split} \int_{\R^d} \varphi(1,x) \d \varrho_1 (x) & - \int_{\R^d} \varphi(0,x) \d \varrho_0 (x)  =\\
& =  \int_0^1 \int_{\R^d} \bra{ \partial_t  \varphi(t,x) + \Tr\bra{ \frac12 a_t(x) \nabla^2 \varphi (t,x)} } \d \varrho_t(x) \d t\\
& \le \int_0^1 \int_{\R^d}\bra{- \cc^*(t,x, \frac12 \nabla^2 \varphi(t,x)) + \Tr\bra{ \frac12 a_t(x) \nabla^2 \varphi }} \d \varrho_t (x) \d t\\
& \le \int_0^1 \int_{\R^d} \cc(t,x, a_t(x)) \d \varrho_t (x) \d t,
\end{split}\end{equation}
and minimizing over the choice of $\varrho$ and $a$ gives
 \[ \int_{\R^d} \varphi(1,x) \d \varrho_1 (x) - \int_{\R^d} \varphi(0,x) \d \varrho_0 (x) \le  \costFP(\varrho_0, \varrho_1).\]

The following result shows that optimizing the left hand side yields equality.

\begin{theorem}[existence and duality]\label{thm:duality}
Let $\Omega = B_r$, for $r>0$ or $\Omega = \R^d$, $p \in (1,+\infty)$ and $\cc:(0,1)\times \R^d \times \Sim_d^+ \to \R \cup \cur{+\infty}$ be $p$-admissible. For every $\mu$, $\nu \in \Prob(\Omega)$, if $\costFP(\mu, \nu) <\infty$, then
\begin{align}\label{eq:duality}
\costFP(\mu, \nu) = \sup \cur{ \int_{\Omega} \varphi(1,x) \d \nu (x) - \int_{\Omega} \varphi(0,x) \d \mu (x)},
\end{align}
the supremum running over all super-solutions $\phi \in C^{1,2}_b([0,1]\times \overline{\Omega})$ to \eqref{eq:pde 2} with boundary condition $\nabla^2 \varphi=0$ if $\Omega = B_r$. Moreover, the  infimum in \eqref{eq:bb-cost-pde} (or in \eqref{eq:bb-cost-process}) is actually a minimum. 
\end{theorem}

\begin{remark}[uniqueness]
Since the map $a \mapsto \cc(t,x, a)$ is strictly convex (in the usual sense), one has that the minimum in \eqref{eq:bb-cost-pde} is unique (see also the proof below). In particular, by Theorem~\ref{thm:markov-projection}, one has that  \eqref{eq:bb-cost-process} is also a minimum, and a minimizer solves the martingale problem associated to the diffusion coefficient of the minimum in \eqref{eq:bb-cost-pde}. However, the problem of  uniqueness (in law) of minimizers in \eqref{eq:bb-cost-process} remains open, as it seems to rely on regularity of minimizers, which implies uniqueness for the martingale problem.
\end{remark}

The proof is an application of the Fenchel-Rockafellar duality theorem, e.g.\ \cite[Theorem 1.9]{Vill09}, following closely \cite[Section 3.2]{Br99}. 

\begin{proof}  We give the proof in the case of $\Omega =B_r$, the case $\Omega = \R^d$ being along the same lines (the only relevant difference is remarked below). First, we may assume that $\int_{\Omega} \ell \cdot x \d \mu = \int_{\Omega} \ell \cdot x \d \nu$, for every $\ell \in \R^d$,  otherwise both terms are $+\infty$: the left hand side because $\mu$ and $\nu$ would not be in convex order, and the right hand side by letting $\varphi (t,x) := \lambda \ell \cdot x$, which solves \eqref{eq:pde 2} (with appropriate boundary conditions) and letting $\lambda \to \pm \infty$, depending on the sign of the difference.

Write then $K := [0,1]\times \overline{B_r}$ and $E=C(K; \R \times \Sim^d)$, which equipped with the uniform norm is a Banach space, with continuous dual $E^* = \meas (K; \R \times \Sim^d)$, and write the dual pairing as $\varrho(F) + \ca(\Phi)$ for $(F, \Phi) \in E$, $(\varrho, \ca) \in E^*$.  In case $\Omega = \R^d$, we let $E = C_b(K; \R \times \Sim^d)$, and since the dual $E^* \supseteq \meas (K; \R \times \Sim^d)$ one has to additionally argue that the linear functional $(\varrho, \ca)$ that  we obtain below is tight, hence induced by a measure.

 Define $\alpha:E\to (-\infty,\infty]$ by
\begin{align*}
\alpha(F,\Phi)=\begin{cases}
                0 & \text{ if  $F(t,x) +\cc^*(t,x, \Phi(t,x))\leq 0$ for every $(t,x) \in K$} \\
\infty & \text{else.}
               \end{cases}
\end{align*}
Being $a^* \mapsto c^*(t,x,a^*)$ convex, it follows that $\alpha$ is a convex function. Its Legendre-Fenchel transform, defined  by
\[
 \alpha^*(\varrho,\ca)=\sup\cur{\varrho(F)+\ca(\Phi)~:~F+\cc^*(\Phi)\leq 0}
\]
is explicitly given by the (strictly) convex function
\begin{equation}\label{eq:alpha*}  \alpha^*(\varrho,\ca) =   \begin{cases}     \int_{K} \cc\bra{ t,x, a(t,x) } \d \varrho(t,x) & \text{if $\varrho \in\meas^+(K)$ and $\ca  = a \varrho$ with $a \ge 0$},\\
\infty & \text{else.}
               \end{cases}
\end{equation}
Indeed, if $\varrho$ is not a positive measure, then we would like to let $\Phi = 0$ and $F =  -\lambda \chi_A$ for some $A$ such that $\varrho(A) < 0$,  and let $\lambda \to \infty$. However, such a choice of $F$ is not immediately possible, but it is sufficient to approximate $\chi_A$, by density of continuous functions in $L^1(\abs{\varrho})$.  Similarly, if $\ca$ is not absolutely continuous with respect to $\varrho$, we would like to let $F = -\cc^*(\lambda I)\chi_A$  (here we use that $\cc^*< \infty$) and $\Phi = \lambda \chi_A$, where $A$ is such that $\varrho(A) = 0$ and $\ca(A) \neq 0$, so that
\[ \alpha^*(\varrho,\ca) \ge \lambda \ca(A) \to +\infty\]
letting $\lambda \to \pm \infty$, depending on the sign of $\ca(A)$. Again, such a choice of $(F, \Phi)$ is not immediately possible, but it is sufficient to approximate $\chi_A$, by density of continuous functions in $L^1(\varrho +\abs{\ca})$. Hence, we may assume $\ca  = a \varrho$ with $a \in L^1(\varrho)$, so that 
\[ \varrho(F)+\ca(\Phi) = \int_K \bra{ F + a \Phi} \d \varrho \le   \int_{K} \bra{a \Phi - \cc^*(\Phi) } \d \varrho \le \int_K \cc\bra{ a } \d \varrho,\]
and by optimizing among $(F, \Phi)$ one obtains \eqref{eq:alpha*}. In the case  $\Omega = \R^d$, at this stage we argue only that if $\alpha^*(\varrho, \ca) <\infty$, then $\varrho$ and $\ca$ are non-negative functionals.  Indeed, if for some $F \ge 0$ one has $\varrho(F) <0$, choosing the pair $(\lambda F, 0)$ with $\lambda \to -\infty$ would give $\alpha^*(\varrho,\ca) = +\infty$, and similarly if for some $\Phi \ge 0$ one has $\ca(\Phi) < 0$, choosing the pair $(0, \lambda \Phi)$ with $\lambda \to -\infty$.

Next, we say that a pair $(F,\Phi)\in E$ is \emph{represented} by $\varphi\in C^{1,2}(K)$ if $F=-\partial_t \phi$, $\Phi=-\frac12\nabla^2\phi$ and $\nabla^2 \phi = 0$ on $[0,1]\times \partial \Omega$. We define $\beta:E\to (-\infty,\infty]$ by
\begin{align*}
 \beta(F,\Phi)=\begin{cases}
                \int_{\R^d} \varphi(0,\cdot) \d \mu - \int_{\R^d} \varphi(1, \cdot) \d \nu  & \text{if  $(F,\Phi)$ is represented by $\varphi\in C^{1,2}(K)$,}\\
\infty & \text{else.}
               \end{cases}
\end{align*}
 Notice first that $\beta$ is well-defined, i.e.\ it does not depend on the choice of $\phi$. Indeed,  if both $\phi$ and $\psi$ represent $(F,\Phi)$, then $\partial_t( \phi - \psi) = F-F = 0$ and $\frac12\nabla^2( \phi - \psi) = \Phi - \Phi=0$ on $K$. It follows that $\phi - \psi = b+ \ell \cdot x$ for some $b \in \R$, $\ell \in \R^d$, and the argument at the beginning of the proof yields $\int_{\R^d} (\ell \cdot x + b) \d (\nu - \mu) = 0$. 

We notice also that the set of represented functions $(F,\Phi) \in E$ is a linear subspace and $\beta$, in the convex set where it is finite, is linear with respect to  the variable $\Phi$. Hence, $\beta$ is convex with Legendre transform
\[
 \beta^*(\varrho,\ca)=\sup\cur{ \varrho(F)+\ca(\Phi) + \int_{\R^d} \varphi(1,\cdot) \d \nu - \int_{\R^d} \varphi(0,\cdot) \d\mu~:~\text{$(F,\Phi)$ is represented}  },
\]
which in fact takes values in $\cur{0, +\infty}$ and it is zero if and only if, for every $\varphi \in C^{1,2} \bra{K }$ with $\nabla^2 \varphi =0$ on $[0,1]\times \partial \Omega$, one has
\begin{align}\label{eq:beta* constraint}
 \int_{K} \partial_t \varphi \d \varrho + \int_{K} \Tr\bra{ \frac12\nabla^2 \varphi  \d \ca} =  \int_{\R^d} \varphi(1, \cdot ) \d \nu - \int_{\R^d} \varphi(0, \cdot ) \d \mu.
\end{align}
{ When $\Omega = \R^d$, the identity above holds provided that we interpret the integrals in the left hand side as duality pairings. However, under the assumption that $\alpha^*(\varrho, \ca) <\infty$ we can use suitable test functions to prove that the non-negative functionals $\varrho$ and $\ca$ are tight, hence measures. Indeed, let $g: [0, \infty) \to [0,1]$ be any smooth non-decreasing function such that $g(x) = 0$ for $|x| \le 1/2$, $g(x) =1$ for $|x| \ge 1$ and $|g'(x)|$, $|g''(x)|  \le 4$. For $M >0$, consider the function $\phi^M(t,x) :=t g( |x|^2/M^2)$, so that the pair
\[ -(F^M, \Phi^M) := \bra{ g( |x|^2/M^2),  t \bra{ g'\bra{\frac{|x|^2}{M^2}} \frac{1}{M^2} \id+ g''\bra{\frac{|x|^2}{M^2}} \frac{{\color{blue}2 }}{M^4} x \otimes x}}\]
is represented by  $\phi^M$. For any $\varepsilon>0$, let $M>0$ be large enough so that
\[ \int_{\R^d} g \bra{ \frac{|x|^2}{M^2}} \d \nu+ | \ca( \Phi^M ) | < \varepsilon,\]
which is possible because of the tightness of $\nu$ combined with the inequality $ g \bra{ \frac{|x|^2}{M^2}} <  \chi_{\cur{|x| > M/2}}$, and the continuity of the linear functional $\ca$ combined with $|\Phi^M(t,x)| \to 0$ uniformly as $M \to \infty$. For any $F \in C_b( [0,1]\times \R^d)$ such that $|F| \le 1$ and $\operatorname{supp}( F) \subseteq \cur{ |x| > M}$, we have that $|F| \le F^M$, hence, being $\varrho$ a non-negative linear functional,  by \eqref{eq:beta* constraint} 
\[ | \varrho (F) | \le \varrho( |F|) \le \int_{K} F^M \d \varrho = \int_{\R^d} g \bra{ \frac{|x|^2}{M^2}} \d \nu - \ca( \Phi^M )  < \varepsilon.\]
This proves that $\varrho$ is tight, hence induced by a measure. To prove that $\ca$ is tight, we introduce the function $G(s) := \int_0^s g(r) \d r$, extended identically $0$ for $s <0$, so that $G'(s) = g(s)$ and $G''(s) = g'(s) \ge 0$. For $M>0$ consider then the function $\phi^M(t,x) :=G( |x|^2-M^2)$, so that the pair
\[  -\frac12 (F^M, \Phi^M) := \bra{0,  2 g\bra{|x|^2 - M^2} \id+ 4 g'\bra{ |x|^2 - M^2} x \otimes x}\]
is represented by $\phi^M$. We notice that $\Phi^M(t,x) \ge 2 \id \chi_{ \cur{|x| \ge M + 1}}$. For any $\varepsilon>0$, let $M>0$ be large enough so that
\[ \int_{\R^d} G \bra{ |x|^2-M^2} \d \nu < \varepsilon,\]
which is possible because $\int_{\R^d} |x|^2 \nu< \infty$ and $ G \bra{|x|^2-M^2} \le  |x|^2 \chi_{\cur{|x| > M}}$. For any $\Phi \in C_b( [0,1]\times \R^d; \Sim^d)$ such that $|\Phi| \le 1$ and $\operatorname{supp} (\Phi) \subseteq \cur{ |x| > M+1}$, we have that $|\Phi| \le \Phi^M$, hence, being $\ca$ a non-negative linear functional,
\[  | \ca (\Phi ) | \le \ca ( |\Phi| ) \le \ca( \Phi^M) = \int_{\R^d} G \bra{ |x|^2-M^2}  \d (\nu - \mu)  < \varepsilon.\]
This proves that  also $\ca$ is tight, hence induced by a measure, thus completing the argument deductions also in the case $\Omega = \R^d$.}

At the point $(-1,0)\in E$, represented by $\phi(t,x) = -t$, we see that $\alpha$ is continuous, for $\cc^*$ is continuous, and $\beta$ is bounded. Therefore the Fenchel-Rockafellar duality \cite[Theorem 1.9]{Vill09} implies
\begin{align}\label{eq:FR duality}
 \inf\cur{\alpha^*(\varrho,\ca)+\beta^*(\varrho,\ca)~:~(\varrho,\ca)\in E^*}= \sup\cur{-\alpha(-F,-\Phi)-\beta(F,\Phi)~:~(F,\Phi)\in E}
\end{align}
and that the left hand side is actually a minimum. Since the right hand side in \eqref{eq:FR duality} is immediately seen to coincide with the right hand side in \eqref{eq:duality}, to conclude we argue that the left hand side above coincides with \eqref{eq:bb-cost-pde}.

Indeed, if $(\varrho, \ca) \in E^*$ is such that $\alpha^*(\varrho, \ca) + \beta^*(\varrho, \ca) < \infty$, we claim that $\varrho := \varrho_t \d t$ where $(\varrho_t)_{ t \in [0,1]} \subseteq \Prob(B_r)$ solves the Fokker-Planck equation $\partial_t \varrho = \Tr(\frac12\nabla^2a\varrho) $ in $(0,1) \times \R^d$. From \eqref{eq:alpha*} and \eqref{eq:beta* constraint}, letting $\varphi(t,x) = \int_0^t g(s) \d s - \int_0^1 g(s) \d s$, with $g \in C([0,1])$, we deduce the identity
\[ \int_{K} g(t) \d \varrho= \int_0^1 g(t)  \d t.\]
Letting $g(t) = 1$, it follows that $\varrho \in \Prob\bra{K}$. Moreover, a density argument implies that the $t$-marginal of $\varrho$ is Lebesgue measure, and by abstract disintegration of measures we have $\varrho =  \varrho_t \d t$ for some Borel curve $(\varrho_t)_{t \in (0,1)} \subseteq \Prob(B_r )$. Moreover, from \eqref{eq:FR duality} we have that the FPE holds in the extended sense of measure-valued solutions, see Section~\ref{sec:notation-basic}. However, since $\alpha^*(\varrho, \ca)< \infty$ implies that $\ca$ is absolutely continuous with respect to $\varrho$, we conclude that the infimum is actually running on the set of weak solutions to \eqref{eq:fpe} (and in particular, the minimum exists in this set). Moreover, by strict convexity of $\alpha$ we deduce that the minimum is unique. 
\end{proof}

\begin{remark}[viscosity solutions]\label{rem:viscosity}
As a general consequence of the comparison principle, Remark~\ref{rem:comparison},   solutions  to \eqref{eq:pde 2}  always increase the right hand side in \eqref{eq:duality} with respect to super-solutions. Indeed, if $\phi, \psi \in C^{1,2}([0,1]\times \overline{B_r})$ are respectively a solution and a super-solution to \eqref{eq:pde 2} (with appropriate boundary conditions) and $\phi(1, x) = \psi(1,x)$ for every $x \in B_r$, then the comparison principle entails $\phi(0,x) \le \psi(0,x)$ for every $x \in \R^d$, hence
\[\int_{\R^d} \phi(1,x) \d \nu - \int_{\R^d} \phi(0,x) \d \mu \ge  \int_{\R^d} \psi(1,x) \d \nu - \int_{\R^d} \psi(0,x) \d \mu.\]
Then, formally, one could restate the duality \eqref{eq:FR duality} by maximizing among solutions, but this comes with the price of introducing viscosity solutions, in order to obtain solutions for any initial datum. This also gives a precise link with the work \cite{TaTo12}, where all the theory relies from the very beginning on viscosity solutions given by ``explicit'' formulas of Hopf-Lax type.
\end{remark}

From Theorem~\ref{thm:duality}, we also obtain sufficient conditions for optimality.

\begin{corollary}\label{cor:phi optimal}
Let $\cc$ be $p$-admissible and $\phi \in C^{1,2}_b([0,1]\times\R^d)$ solve \eqref{eq:pde 2}, set
\begin{equation}
\label{eq:a-optimal-phi} 
 a_t(x):=\nabla_{u} \cc^*(t,x,\frac12\nabla^2\phi(t,x)) \quad \text{for $(t,x) \in (0,1) \times \R^d$,}\end{equation}
and let $(\varrho_t)_{t\in [0,1]}\subset \mathcal{P}(\R^d)$ be a solution to \eqref{eq:fpe}. Then, $(\varrho_t)_{t\in [0,1]}\subset \mathcal{P}(\R^d)$ is a minimizer in \eqref{eq:bb-cost-pde}, i.e.\ for every $s \le t \in [0,1]$,
\[ \costFP(\varrho_s, \varrho_t) = \int_0^1 \int_{\R^d} \cc(r,x, a_r(x)) \d \varrho_r (x) \d r.\]
\end{corollary}

\begin{proof}
Indeed, for any super-solution $\psi \in C^{1,2}([0,1]\times\R^d)$ to \eqref{eq:pde 2}, arguing as in \eqref{eq:duality-inequality}, one has
\[\int_{\R^d} \psi(1,x) \d \varrho_1 (x) - \int_{\R^d} \psi(0,x) \d \varrho_0 (x) \le \int_0^1 \int_{\R^d} \cc(t,x, a_t(x)) \d \varrho_t (x) \d t. \]
with equality if $\psi = \phi$, hence Theorem~\ref{thm:duality} yields the thesis.
\end{proof}

\begin{example}[transporting Gaussian measures]
 Let $\gamma_{0,Q} \in \Prob(\R^d)$ be a  Gaussian measure with mean $0$ and covariance matrix $Q$. For any $\mu \in\mathcal P(\R^d)$, let $X_0$ be a random variable with law $\mu$, define $\nu := \mu * \gamma_{0,Q}$ and let $(B_t)_{t \in [0,1]}$ be standard a $d$-dimensional Brownian motion (with $B_0 = 0$), independent of $X_0$.  Then, the martingale $(X_0 + \sqrt{Q}B_t)_{t\in [0,1]}$ is an optimizer for any cost function of the form $\cc(t,x, a)=\cc(a)$. Indeed, by \eqref{eq:optimality-c-c-star}, there exists $R \in \Sim^d_+$ such that $Q = \nabla_u \cc^*(\frac{R}{2})$ and letting
\[ \varphi(t,x) := -t \cc^*\bra{\frac{R}{2}} + \frac 1 2 (Rx) \cdot x , \quad \text{for $(t,x) \in [0,1]\times \R^d$,}\]
 one has that $\varphi$ solves \eqref{eq:pde 2}. We are in a position to apply Corollary~\ref{cor:phi optimal}, which implies that the curve $(\mu * \gamma_{0, t Q})_{t \in [0,1]}$, which coincides with that of the $1$-marginals of $(X_0 + \sqrt{Q}B_t)_{t \in [0,1]}$ is a minimizer.
\end{example}

\begin{remark}
One can  formally write an equation for the optimal diffusion coefficient $a$, using \eqref{eq:pde 2}, \eqref{eq:a-optimal-phi} or its dual relation $\frac{1}{2} \nabla^2 \varphi = \nabla_a \cc\bra{t,x, a(t,x)}$ A differentiation of the latter yields that $a:[0,1]\times\R^d\to\Sim_d^+$ satisfies 
\begin{align}\label{eq:a optimal}
 \partial_t a_t= -\bra{\nabla_a^2\cc(t,x,a(t,x))}^{-1}\bra{ \frac{1}{2} \nabla_x^2 (\cc^*(t,x,u(t,x)) +( \partial_t\nabla_a \cc)(t,x,a(t,x))} ,
\end{align}
where $u (t,x) := \nabla_a \cc\bra{t,x, a(t,x)}$. To make this argument rigorous, however, one has to assume that $\cc$, $\cc^*$ are smooth enough, $\nabla_a^2\cc(t,x,\cdot)$ is invertible. Moreover, to argue that if $a$ satisfies \eqref{eq:a optimal} and $\varrho = (\varrho_t)_{t \in [0,1]}$ solves \eqref{eq:fpe}, then $\varrho$ is a minimizer for the problem of transporting $\varrho_0$ to $\varrho_1$, one has to prove that $u(t,x) = \frac{1}{2} \nabla^2 \varphi(t,x)$ for some $\phi \in C^{1,2}_b([0,1]\times\R^d)$. For $d=1$ this becomes easier as we show in the following section.
\end{remark}

\section{One-dimensional case}\label{sec:one-dim}

In this section we specialize to the case $d=1$.  First, we investigate sufficient conditions, besides the necessary convex ordering between $\mu$ and $\nu$, ensuring that $\costM(\mu, \nu)$ is finite, if $\cc$ has $p$-growth for some $p \ge 1$. When $p=1$, It\^o's isometry implies that, for any martingale $X$,
\[ \int_0^1 \esp{ \cc( t, X_t, \dot{\ang{ X}}_t)} \d t \le \lambda \esp{ \int_0^1  \dot{\ang{ X}}_t \d t}  =\lambda \esp{ (X_1 - X_0)^2 } < \infty\]
using the assumptions on the second moments of $\mu$ and $\nu$. For  $p>1$, the Burkholder-Davis-Gundy inequalities in combination with Jensen's inequality provide a lower bound, hence they seem to be of no use. Nevertheless, relying on the solution to the Skorokhod embedding problem, we have the following result.

\begin{proposition}
 Let $\cc$ have $p$-growth, for some $p>1$, let $\mu$ and $\nu$ be in convex order with $\int_{\R} |x|^q\ \d \nu(x)<\infty$ for some $q>2p$. Then, there is a martingale $X\in\ACQV{}$ with $(X_0)_\sharp \prob = \mu$, $(X_1)_\sharp \prob = \nu$ and
$$ \int_0^1\E\left[ \cc\bra{t, X_t, \dot{\ang{X}}_t} \right] dt<\infty\ .$$
In particular, $\costM(\mu, \nu)<\infty$.
\end{proposition}

\begin{proof}
 It is sufficient to show the thesis with $\cc(t,x,a) = |a|^p$. Let $\tau$ be any solution to the Skorokhod embedding problem (see \cite{Ob04} and \cite{Ho11} for comprehensive surveys) for $\nu$ with start in $\mu$ such that $(B_{t\wedge \tau})_{t\geq 0}$ is uniformly integrable, i.e.\ $B_0\sim \mu, B_\tau\sim \nu$. The assumption on the $q$-th moment of $\nu$, together with the Burkholder-Davis-Gundy inequality and the uniform integrability of $( B_{t\wedge \tau})_{t\geq 0}$, imply that $\E[\tau^{q/2}]\leq \lambda \int \abs{x}^q \d \nu(x)<\infty$, for some $\lambda = \lambda(q)$. Next, we perform a change of time introducing the martingale $(X_t)_{t\in [0,1]}$,
$$X_t := B_{\tau \wedge \beta_t}$$
with $\beta_t := \left(\frac{t}{1-t}\right)^{1/r}$, for some $r \in \left( 0, \frac{q-2p}{2p-2}\right]$. Clearly, $X\in\ACQV{}$ with $\dot{\ang{X}}_t=\chi_{\cur{\beta_t \leq \tau }} \beta'_t=: a_t.$ Moreover, we can calculate
\begin{equation*}\begin{split} 
 \int_0^1 a_t^{p} \d t&  = \int_0^{\beta^{-1}(\tau)}(\beta'_t)^{p-1} \beta'_t \d t = \int_0^\tau \bra{(\beta^{-1})'_s}^{1-p} \d s\\
 &= \int_0^\tau \frac 1{r^{p-1}} s^{(1-r)(p-1)}(1+s^r)^{2(p-1)} \d s \\
 & \leq  \lambda(r,p) \left(\tau^{(1-r)(p-1)+1} + \tau^{p} + \tau^{(1+r)(p-1)+1}\right).
 \end{split}\end{equation*}
The choice of $r$ ensures that $(1+r)(p-1)+1\leq q/2$ so that
\[ \E\sqa{ \int_0^1 a_t^{p}} \d t \leq  \lambda(r,p) \E \sqa{ 1+\tau^{q/2}} \leq \lambda \bra{1+\int_{\R^d} \abs{x}^q d\nu} <\infty,\]
 where we  applied the Burkholder-Davis-Gundy inequality once more.\end{proof}

When  both $\mu$ and $\nu$ have densities with respect to the Lebesgue measure, one can rely on the following variant of the Dacorogna-Moser interpolation technique \cite{DaMo90}. See also \cite[Theorem 1.1]{doring_skorokhod_2017} for recent applications of similar ideas, extending the results in \cite{Du94}, to the Skorokhod embedding problem for Levy processes.

\begin{proposition}[Dacorogna-Moser interpolation]\label{prop:dac-mos}
Let $\mu  = m (x) \d x$, $\nu  = n (x) \d x \in \Prob(\R^d)$ have strictly positive densities, be in convex order and let $f  := k \conv (n-m)$, with $k(x) = x^+$. 
Then, $f \ge 0$ and, if $\varrho_t := (1-t) \mu + t \nu$, $a_t :=  2f/\bra{(1-t)m + t n}$, for $t \in [0,1]$, then \eqref{eq:fpe} holds. Moreover, for any $p \ge 1$,
\begin{equation}\label{eq:dac-mos-upper} \int_0^1 \int_{\R} \abs{a_t}^p \d \varrho_t  \d t \le  \norm{n-m}_{L^1}^{p-1} \int_{\R} \abs{n(x)-m(x)} \int_0^x M_{p}\bra{ m(y), n(y) } |y-x|^p \d y \d x, \end{equation}
where, for $u$, $v >0$,
\[ M_p(u,v) :=  \int_0^1 ((1-t)u +tv)^{1-p} = \begin{cases} \frac{ v^p - u^p}{(p-2)(v-u)} & \quad \text{if $p \neq 2$,}\\
 \frac{\log (v/u)}{v-u} & \quad \text{ if $p=2$.}
 \end{cases}\]
\end{proposition}

\begin{proof}
The fact that $f(x) \ge 0$, for $x \in \R$, follows from the assumption of convex ordering, writing
\[ f(x) =  \int_{\R} k(y-x) (n(y)-m(y) ) \d y = \int_{\R} (y-x)^+ \d \nu(y) - \int_{\R} (y-x)^+ \d \mu (y)\]
and using the convexity of $y \mapsto (y-x)^+$. To show that \eqref{eq:fpe} holds, we use instead the fact that $x^+$ is a fundamental solution to the Laplace equation. Hence, we formally write 
\[ \partial_t \varrho= \nu - \mu = \Delta f = \Delta \bra{\frac12 a \varrho}.\]
Rigorously, the identity 
\[ \Delta f = n-m \]
holds in duality with functions in $C^2(\R)$. Let then $\varphi \in C^{1,2}([0,1] \times \R^d)$ and argue by duality,
\[\begin{split} \frac{\d}{\d t} \int_{\R} \varphi (t, x)  \d \varrho_t(x) &  = \int_{\R} \partial_t \varphi (t,x) \d \varrho_t(x)  + \int_{\R} \varphi (t,x) (n(x)-m(x)) \d x  \\
& = \int_{\R} \partial_t \varphi (t,x) \d \varrho_t(x)  + \int_{\R} \varphi (t,x) \Delta f \d x\\
& =  \int_{\R} \partial_t \varphi (t,x) \d \varrho_t(x)  + \int_{\R} \frac12 a_t(x)\bra{  \Delta \varphi} (t,x)  \bra{ (1-t)m(x) + tn(x)} \d x\\
& =  \int_{\R} \partial_t \varphi (t,x) \d \varrho_t(x)  + \int_{\R} \frac12 a_t(x)\bra{  \Delta \varphi} (t,x) \d \varrho_t(x),\\
\end{split}\]
hence integrating with respect to $t \in (0,1)$, we obtain \eqref{eq:fpe-test}.

To show \eqref{eq:dac-mos-upper}, we notice first that, since both $m$ and $n$ are probability densities and, by the convex order assumption on $\mu$ and $\nu$, $\int_\R y m(y) \d y = \int_{\R} y n(y) \d y$, we have the identity, with $k'(x) = \chi_{\cur{x>0}}$,
\[ f(x) = \int_{\R} \bra{ k(x-y) - k(x) - k'(x) y} \bra{n(y)-m(y)} \d y.\]
Hence, by H\"older's inequality,
\[ \begin{split} \abs{f(x)}^p & =  \abs{ \int_{\R}\bra{k(x-y) - k(x) - k'(x) y} \bra{m(y)-n(y)} \d y}^{p}\\
&  \le  \norm{m-n}_{L^1}^{p-1} \int_{\R}\abs{k(x-y) - k(x) - k'(x) y}^p |n(y)-m(y)| \d y. \end{split}\]
We then have 
\[ \begin{split} 
\int_0^1 \int_{\R} \abs{a_t}^p \d \varrho_t  \d t  &= 2^p\int_{\R} \abs{f}^p(x) \int_0^1 \bra{ (1-t)m(x)+tn(x))}^{p-1} \d t \\
& = 2^p \int_{\R} \abs{f}^p(x) M_p\bra{m(x), n(x)} \d x\\
& \le 2^p \norm{n-m}_{L^1}^{p-1} \int_{\R} |n(y)-m(y)| \int_0^y M_p\bra{m(x), n(x)} |y-x| \d y \d x,
\end{split}\]
since $\abs{k(x-y) - k(x) - k'(x) y} = |y-x| \bra{ \chi_{\cur{y<x<0}} + \chi_{\cur{0<x<y}}}$.
\end{proof}

As an application of the technique of Proposition~\ref{prop:dac-mos} and Theorem~\ref{thm:sp}, we may also obtain a PDE proof of the following fundamental result \cite{St65}.

\begin{corollary}[Strassen's theorem]\label{cor:strassen}
If $\mu$, $\nu \in \Prob(\R)$ are in convex order, then there exists a (discrete-time) martingale $(X_0, X_1)$ such that $(X_0)_\sharp \prob =\mu$, $(X_1)_\sharp \prob = \nu$.
\end{corollary}

\begin{proof}
For $\eps\in (0,1]$, let $\sigma^\eps(x)$ be a smooth, positive everywhere, mollification kernel on $\R$ converging to $\delta_0$ as $\eps\to 0$ (e.g.\ the heat kernel) and define $\mu^\eps := \mu\conv \sigma^\eps $, $\nu^\eps := \nu \conv \sigma^\eps$, which are in convex order and have strictly positive densities $m^\eps$, $n^\eps$ with respect to the Lebesgue measure, and let $\varrho^\eps$, $a^\eps$ be as in Proposition~\ref{prop:dac-mos}.  By Theorem~\ref{thm:sp}, there exist solutions $(X^\eps_t)_{t \in [0,1]}$ to the martingale problem associated with $a_t^\eps \Delta$,  with $1$-marginals $(\varrho_t^\eps)_{t \in [0,1]}$, which are in particular martingales. Moreover, since the marginals at time $0$ and time $1$ converge as $\eps\to 0$ (respectively to $\mu$ and $\nu$), we have that the family of the joint laws of the martingales $(X^\eps_0, X^\eps_1)$, for $\varepsilon \in (0,1]$ is tight, hence pre-compact. Any limit point provides a martingale as required.
\end{proof}

Notice however that, in general, with this technique one cannot find an interpolating martingale $(X_t)_{t\in [0,1]}$ with continuous paths. For example, when $\mu = \delta_0$, $\nu = \frac 12 \bra{\delta_{-1}+ \delta_1}$, one would obtain a martingale with marginals $\varrho_t = (1-t) \mu + t \nu$, for $t \in [0,1]$, which must have discontinuous paths.

\vspace{1em}

In the remaining part of this section we discuss duality and optimizers. A crucial aspect in $d=1$ is that $\nabla^2 = \partial^2_x =  \Delta$, so that equation \eqref{eq:pde 2} can be written as
\begin{equation}\label{eq:dual-filtration} \partial_t \phi (t,x) = - \cc^*(t,x, \frac12\Delta \phi(t,x) ). \end{equation}
A natural (but formal) idea is then to operate with $\frac12\Delta$ on both sides, so that the variable $u := \frac12\Delta \phi$ solves the backwards generalized porous medium type equation
\begin{equation}\label{eq:pme}
\partial_t u(t,x) =-  \Delta \frac{\cc^*}{2}\bra{t,x, u(t,x)}.\tag{$\mathsf{PME}$}
\end{equation}
The boundary conditions $\Delta \varphi = 0$ on $\partial B_r$ become Dirichlet boundary conditions for $u$. 

Let us point out that, to make a full connection with the variational problems, some difficulties appear because the duality result, Theorem~\ref{thm:duality}, is formulated in terms of classical super-solutions, and (even formally) we must use solutions to perform the change of variable $u = \frac12\Delta \varphi$. To the authors' knowledge, even introducing viscosity solutions (see also Remark~\ref{rem:viscosity}) may not be useful, since they may be not sufficiently regular to provide any weak notion of Laplacian.

Nevertheless, we rely on the theory of porous medium equations \cite{Va07} to obtain solutions $u$ to \eqref{eq:pme} and recover solutions to \eqref{eq:dual-filtration} via integration.  We also notice that in the literature of porous medium equations the variable $\varphi = 2\Delta^{-1} u$ is called a potential and \eqref{eq:dual-filtration} is also called \emph{dual filtration equation}.

\begin{definition}[weak solutions to \eqref{eq:pme}]
Let $r >0 $. We say that $u \in C([0,1]\times [-r,r])$ is a solution to \eqref{eq:pme} if, for every $g \in C^{1,2}_c ((0,1) \times (-r,r))$ it holds
\[  \int_0^1 \int_{-r}^r  u(t,x) \partial_t g (t,x)\d x \d t = \int_0^1 \int_{-r}^r  \frac12\cc^*\bra{t,x, u(t,x)} \Delta g(t,x) \d x \d t .\]
\end{definition}

The one-dimensional theory of such porous medium type equations is well understood, at least in the case of $\cc^*(t,x, u) = 2|u|^q$, see \cite[Chapter 15]{Va07}. We rely on such results to prove the following theorem.

\begin{theorem}[existence of solutions to \eqref{eq:dual-filtration}]\label{thm:existence-pme-hjb}
Let $\cc^*(t,x, u) = 2(u^+)^q$, $q >1$. $r \in [0, \infty)$, $\bar u \in C([-r,r])$, $\bar{u} \ge 0$ and $\bar{u}(-r) = \bar{u}(r) = 0$. Then, there exists a unique $\varphi \in C^{1,2}_b( [0,1]\times [-r, r])$ solving \eqref{eq:pde 2} in $(0,1) \times (-r, r)$ with boundary condition $\Delta \varphi = 0$ on $(0,1) \times \cur{-r, r}$ and $\Delta \varphi_1 = \bar u$. Moreover, one has $\Delta \varphi(t,x) \ge 0$ for $(t,x) \in (0,1) \times (-r, r)$ and for every $t \in [0,1)$, $\Delta \varphi(t,x)$ is Lipschitz continuous.
\end{theorem} 

\begin{proof}
By the results in \cite[Chapter 15]{Va07},  there exists a weak solution $u$ to \eqref{eq:pme} with $\cc^*(t,x,u) = 2|u|^q$, Dirichlet boundary conditions and $u_1 = \bar{u}$. Moreover, the maximum principle ensures that $u_t \ge 0$ for $t \in [0,1]$ hence $u$ is also a solution with respect to $\cc^*(t,x,u) = 2(u^+)^q$. Moreover, for $t\in [0,1)$, $u_t(x)$ is Lipschitz continuous is space and in time \cite[Theorem 15.6]{Va07}. For every $t \in [0,1]$ let $\varphi_t$ solve $\frac12\Delta \varphi_t = u_t$ in $(-r,r)$  with Dirichlet boundary conditions $\varphi_t (x) = 0$ for $x \in \cur{-r, r}$. Then, $\varphi \in C^{1,2}([0,1)\times[-r,r])$ and one has the identity
\[ \frac12\Delta \partial_t \varphi = \partial_t \frac12\Delta \varphi = \partial_t u = - \frac12\Delta 2|u|^q.\]
Since both  $\partial_t \varphi$ and $-2|u|^q$ agree on the boundary (both are null), we deduce that $\varphi$ solves \eqref{eq:pde 2}. Moreover, since the right hand side in \eqref{eq:pde 2} is continuous up to $t =1$, we deduce that $\varphi \in C^{1,2}_b([0,1)\times[-r,r])$. 
\end{proof}

\begin{remark}[Pressure equation]\label{rem:pressure}
The connection between \eqref{eq:pme} and \eqref{eq:dual-filtration} is even more intriguing if one looks directly at the equation for the optimal diffusion coefficient  \eqref{eq:a optimal}. Indeed, letting $\cc^*(t,x,u) = 2(u^+)^q$, by duality $\cc(t,x,a) = a^p/\bra{p (2q)^{p/q}}$ with $p=q/(q-1)$, and one gets
\begin{align}\label{eq:pressure}
 \partial_t a(t,x) = - \frac{1}{2p } \frac{\Delta a^p(t,x)}{a^{p-2}(t,x)} = - \frac{1}{2} \bra{ a(t,x)\Delta a(t,x) + (p-1)(\partial_x a(t,x))^2 },\tag{$\mathsf{PMPE}$}
\end{align}
which is precisely the \emph{pressure equation} associated to \eqref{eq:pme}. The interplay between equation \eqref{eq:pme} and \eqref{eq:pressure} is very well understood, and may be useful to provide examples, e.g.\ using explicit solutions such as Barenblatt profiles. Let us also notice the regularity theory for \eqref{eq:pressure} yields Lipschitz continuity of $x \mapsto a(t, x)$, hence $1/2$-H\"older continuity of $\sigma(t, x) = \sqrt{a(t,x)}$.  Then, the  Watanabe criterion \cite[Chapter IX, \textsection 3]{ReYo99} provides uniqueness of solutions to the martingale problem. Thus, minimizers of \eqref{eq:bb-cost-process} obtained via Corollary~\ref{cor:phi optimal}    are unique in law.
\end{remark}

One can even allow for  ``explosive'' initial data $\bar{u}$ in Theorem~\ref{thm:existence-pme-hjb}, leading to non-trivial interpolations, as the next example shows.

\begin{example}\label{ex:friendly giant}
Let  $\nu=\frac{1}{2}(\delta_{-1}+\delta_1)$ and $\mu$ in convex order with respect to $\nu$. Let $\cc^*(t,x,u)=2(u^+)^q$  for $q\in (1, \infty)$ and $u=(u_t)_{t\in[0,1]}$ be the solution to the backwards porous medium equation 
$$\partial_t u=-\Delta u^q$$
with terminal condition $u_1=\infty\chi_{(-1,1)}$. More precisely, we let $u$ be the so-called friendly giant (backward), so that \cite[Theorem 5.20]{Va07} gives $u_t(x)=(1-t)^{-\frac{1}{q-1}}g(x)$ where $g$ is the unique (positive everywhere) solution to
$$ \Delta g^q + \frac{1}{q-1}g=0,\quad g^q \in H^1_0(-1,1),$$
 $H^1_0(-1,1)$ being the usual first order Sobolev space of square-integrable functions on $(-1,1)$, with square integrable derivative, and null trace at the boundary.
Defining $a_t=qu_t^{q-1}=q \frac{1}{1-t} g^{q-1}$, which is the corresponding pressure variable, then $a$ solves \eqref{eq:a optimal}. Let $\varrho=(\varrho_t)_{t\in [0,1)}$ be a solution to the corresponding Fokker-Planck equation with initial condition $\varrho_0=\mu$. For any $t<1$, a solution exists since $a$ is bounded and continuous. 

We argue that necessarily $\lim_{t \uparrow1} \rho_t=\nu$. Let $X=(X_t)_{t\in [0,1)}$ be a continuous martingale with marginals $(\varrho_t)_t$ and $\dot{\qv{X}_t}=a_t(X_t)$ on some probability space $(\Omega,\cA,\P,(\F_t)_{t\in [0,1]})$. Observe that $X$ can only stop diffusing  at the boundary $\cur{-1,1}$, for $a_t(x)>0$ on $[0,1)\times(-1,1)$. For $y\in (0,1)$ let 
$$\tau_y=\inf\cur{t\geq 0 : \abs{X_t}\geq y}.$$
We claim that $\P\bra{\tau_y<1}=1$. Indeed, put $g^*(y)=\inf\cur{g(x) : x\in [-y,y]}>0$ by positivity of $g$ inside $(-1,1)$. By the Dambis-Dubins-Schwarz Theorem, possibly enlarging our probability space, $X$ is a time change of Brownian motion $B$, i.e.\ $X_t=B_{\qv{X}_t}.$ Hence, for any $s<1$ we have
\[ \begin{split}
 \P\bra{\tau_y>s} & = \P\bra{\max_{0\leq r\leq s} \abs{X_r} <y} = \P\bra{\max_{0\leq r\leq s} \abs{B_{\qv{X}_r}}<y}\\
 & \leq \P\bra{\max_{0\leq r\leq s} \abs{B_{\int_0^r\frac{1}{1-t}g^*(y)dt}}<y},
\end{split}\]
which goes to zero for $s$ tending to $1$. Hence, for any $y\in (0,1)$ we have $\P\bra{\tau_y<1}=1$ which implies by continuity of $X$ that $\P\bra{\tau_1\leq 1}=1$. This in turn implies our claim. As a consequence of Corollary \ref{cor:phi optimal},  $X$ is a minimizer in \eqref{eq:bb-cost-process} with $\cc(a) = a^p/\bra{p (2q)^{p/q}}$, $p = \frac{q}{q-1}$.
\end{example}

\begin{remark}[localization does not preserve optimality]\label{rem:local}
Differently from the classical optimal transport, ``localization''  does not preserve optimality, in general. Indeed, considering the previous example with $\mu=\delta_0$, by stopping $X$ upon leaving the interval $(-y,y)$ will force the law of $X^\tau_1$ to be $\frac{1}{2}(\delta_{-y}+ \delta_y)$. However, $\dot{\qv{X^\tau_t}}$ is not induced by the corresponding friendly giant on $(-y, y)$, hence it cannot be optimal.
\end{remark}

\section{Conclusion}\label{sec:conclusion}
In this article we introduced a class of Benamou-Brenier type martingale optimal transport problems via a weak length relaxation procedure of discrete martingale transport problems. By linking this class of optimization problems to the classical field of martingale problems and Fokker-Planck equations, we established an equivalent PDE formulation using only marginal probabilities and diffusion coefficients, possibly gaining a complexity reduction.

This approach as well as our results lead to a number of interesting questions that we leave for future work. 

The first one is whether a stronger ``relaxation'' result than Theorem~\ref{thm:length relax} holds, similarly to the construction of a length metric induced by a distance. A statement could be as follows. Given $\cc$ as in Theorem~\ref{thm:length relax}, for any partition $\pi=\{0=t_0<\ldots<t_n=1\} \subseteq [0,1]$, introduce the rescaled cost $\cc^{i}(x,y) = \cc( (y-x)/ \sqrt{t_{i} - t_{i-1}})$ and, for $\mu$, $\nu \in \Prob(\R^d)$, let $\Wd^i(\mu, \nu)$ be  the associated martingale optimal transport cost as in \eqref{eq:discrete mot}. Define 
\[ \cc^{\pi} (\mu, \nu) := \inf\cur{ \sum_{i=1}^n  \Wd^{i}(\varrho_{t_{i-1}},\varrho_{t_i}) (t_i - t_{i-1}) : \text{ $(\varrho_{t_i})_{t_i \in \pi}  \subseteq \Prob(\R^d)$, $\varrho_0 = \mu$, $\varrho_1 = \nu$.}}.\]
Then, $\liminf_{ \norm{\pi} \to 0} \cc^{\pi} (\mu, \nu ) = \bar{\cc}_{BB}(\mu, \nu)$, where $\bar{\cc}(a)$ is given as in \eqref{ca-intro}. It seems possible to obtain the inequality $\le$, up to approximating any $X \in \ACQV{}$ by martingales as in Theorem~\ref{thm:length relax}. The validity of the converse inequality appears to be more difficult. Such a result would provide a closer connection between the discrete- and continuous-time problems.

As a second problem, one could ask whether in the duality \eqref{eq:duality} the supremum is actually a maximum, after a suitable relaxation of the notion of solution to \eqref{eq:pde 2}, e.g.\ arguing with viscosity solutions. Then, natural questions such as regularity of such optimal ``potentials'' should be addressed, possibly leading to a deeper understanding of the structure of optimizers of the primal problem. 

A third problem is the  numerical study of the variational problems,  exploiting the PDE formulation to reduce complexity (and make a comparison with the examples in \cite{TaTo12}). In this direction, we quote the recent preprint \cite{Guo2017} for a numerical study of the one-dimensional problem via duality.

Finally, we mention that in the one-dimensional case, after the change of variable $u =\frac12\Delta \varphi$, the dual problem \eqref{eq:duality} in Theorem \ref{thm:duality} seems to be equivalent to 
\[\sup \cur{ \int_{\R} u(1,x) \pi_\nu(x) \d x - \int_\R u(0,x) \pi_{\mu}(x) \d x }, 
\]
the supremum running over all solutions to \eqref{eq:pme}, and $\pi_\varrho(x) := \int_{\R} |x-y| \d \varrho (x)$ being the one-dimensional Newtonian potential of a measure $\varrho$. Rigorously, this alternative duality is equivalent to Theorem \ref{thm:duality} if we ask $u$ to be a weak super-solution to \eqref{eq:pme} in duality against convex functions only; i.e.\ if we restrict the class of test functions for \eqref{eq:pme} to  convex functions. Is such an extremely weak notion to \eqref{eq:pme} sufficient to obtain a reasonable theory of well-posedness? 
This alternative formulation of the dual problem is also very natural, as it takes the irreducible components into account (see \cite{BeJu16} for the definition and use of irreducible components for martingale transport, and \cite{BeNuTo16, BeLiOb17} for an application to duality). This can be seen by writing
\begin{align*}
&\int_{\R} u(1,x) \pi_\nu(x) \d x - \int_\R u(0,x) \pi_{\mu}(x) \d x\\
 =& \int_\R u(1,x)(\pi_\nu(x)-\pi_\mu(x))dx + \int_\R (u(1,x)-u(0,x)) \pi_\mu(x) dx.
\end{align*}
The difference of the potentials in the first integral suggests --- what is known from a nice probabilistic argument (see \cite[Appendix]{BeJu16} and \cite{BeNuTo16}) --- that the initial condition $u(1,\cdot)$ should be independently defined on each open component of $\{\pi_\nu-\pi_\mu\}$. Developing a good understanding of this phenomenon from the PDE point of view could be particularly interesting for the higher dimensional case, which is far more complicated since it is unclear which functional should replace the Newtonian potential, see \cite{DeTo17, ObSi17}.

\appendix
\section{Proof of Theorem~\ref{thm:length relax}}\label{app:proof}

We consider only \eqref{eq:length relax}, the proof of \eqref{eq:length relax two} being similar. By continuity of $\dot{\ang{ X}}$, the Riemann sums computed on $\pi = \cur{t_0 =0 < \ldots < t_n=1} \subseteq [0,1]$,
\[ \sum_{i=1}^n \E\left[ \cc\bra{\textstyle{\sqrt{\dot{\qv{X}}}_{t_{i-1}} }\, Z} \right] (t_i- t_{i-1}), \]
converge to the right hand side  of \eqref{eq:length relax} as $\norm{\pi} \to 0$. Therefore, it suffices to prove that 
\begin{equation}\label{eq:proof-relaxation-error} \lim_{\nor{\pi}\to 0} \sum_{i=1}^n \bra{ \E\sqa{ \cc\bra{\frac{X_{t_i}-X_{t_{i-1}}}{ \sqrt{t_i- t_{i-1}} }}} -  \E\left[ \cc\bra{\textstyle{\sqrt{\dot{\qv{X}} }_{t_{i-1}}}\, Z} \right] }(t_i- t_{i-1})  = 0. \end{equation}

To this aim, we use  \cite[Theorem~3.9]{ReYo99}, so that there exists a predictable process $(\sigma_t)_{t\in [0,1]}$ with values in $\Sim^d_+$ such that
\[ X_t = X_0+\int_0^t \sigma_s \d W_s, \quad \text{for $t \in [0,1]$, $\prob$-a.s.,}\]
where $(W_t)_{t \in [0,1]}$ is a $d$-dimensional Wiener process, possibly on an enlarged probability space. In fact, the proof of \cite[Theorem~3.9]{ReYo99} gives the identity  $\sigma_t = \textstyle{\sqrt{\dot{\qv{X}}}}_t U_t$ for some  predictable process $(U_t)_{t \in [0,1]}$ with values in  $d \times d$ orthogonal matrices. Up to replacing  $(W_t)_{t \in [0,1]}$ with the Wiener process   $\int_0^t U_s \d W_s$, we can assume that $U_t$ is the identity matrix for $t \in [0,1]$, and $\sigma = \textstyle{\sqrt{\dot{\qv{X}}}}$.

In this situation, for $t >0$, one has, by the  Burkholder-Davis-Gundy inequality 
\begin{equation}\label{eq:bdg-bound-1} \E\sqa{ \abs{ \frac{X_t -X_0}{\sqrt{t}} }^{2p} }  \le \lambda_p  \E  \sqa{ \abs{ \frac 1t \int_0^t  |\sigma_s|^2 ds}^{p}} \le 	\lambda \E  \sqa{ \frac{1}{t} \int_0^t  \abs{\sigma_s}^{2p} \d s },\end{equation}
(here and below, $\lambda$ represent different constants, possibly changing line to line). Similarly, starting from the identity
\[ X_t -X_0 - \sigma_0 W_t    =  \int_0^t \bra{ \sigma_s -\sigma_0 } \d W_s\]
we obtain
\begin{align}\label{eq:bdg-bound-2}
\E \sqa{ \abs{ \frac{X_t-X_0-\sigma_0 W_t}{\sqrt{t}}}^{2p}  }
& \leq  \lambda \E  \sqa{ \abs{ \frac 1t \int_0^t \abs{ \sigma_s - \sigma_0}^2 ds}^{p}}	\\ \nonumber
\end{align}
where the second inequality follows by Jensen's inequality and  the inequality $\abs{ a^{1/2} - b^{1/2} } \le \lambda \abs{ a-b }^{1/2}$, for  $a$, $b \in \Sim^d_+$.

 Using the assumption $|\cc(y) - \cc(x)| \le \lambda(1+ |x|^{2p-1} + |y|^{2p-1}) |y-x|$, with  $x = (X_t-X_0)t^{-1/2}$ and $y = \sigma_0 W_t t^{-1/2}$,  we have that
\[ \begin{split}
 \bigg| &  \E \sqa{  \cc \bra{ \frac{X_t -X_0}{\sqrt{t}} } }     - \E \sqa{   \cc \bra{ \frac{ \sigma_0 W_t}{\sqrt{t} } }} \bigg|  \le  \\  
& \le  \E \sqa{ A  \abs{ \frac{X_t -X_0 -  \sigma_0 W_t}{\sqrt{t}} }} \\
& \quad\quad\quad \text{with $A := \lambda(1+ |(X_t-X_0)t^{-1/2}|^{2p-1} +  |\sigma_0 W_t t^{-1/2}|^{2p-1})$,} \\
& \le \eps \E  \sqa{A^{2p/(2p-1)}} +  \lambda(\eps) \E  \sqa{ \abs{ \frac{X_t-X_0- \sigma_0 W_t}{\sqrt{t}}} ^{2p} }  \\
&  \quad \quad\quad \text{by Young's inequality (with $\eps>0$),}\\  
& \le \eps \E  \sqa{ 1+\frac{1}{t} \int_0^t  \abs{\sigma_s}^{2p} \d s + \abs{\sigma_0 }^{2p}|W_tt^{-1/2}|^{2p} }  +  
\lambda(\eps)   \E  \sqa{ \frac{1}{t} \int_0^t  \abs{\sigma^2_s -\sigma^2_0}^{p} \d s } \\
&  \quad\quad\quad\text{by \eqref{eq:bdg-bound-1} and \eqref{eq:bdg-bound-2}.}
\end{split}\]
Moreover, since $W_t t^{-1/2}$ is a  $d$-dimensional standard Gaussian, independent of $\cF_0$, we have
\[ \E\sqa{ \cc \bra{ \sigma_0 W_t t^{-1/2} }}  =  \E\sqa{ \cc \bra{ \sigma_0 
Z}}.\] 

Then, for a given partition $\pi= \cur{0=t_0 < \ldots < t_n=1}$, for $i \in \cur{1, \ldots, n}$, we apply \eqref{eq:bdg-bound-1} and \eqref{eq:bdg-bound-2} to each martingale $X^i_s := X_{(1-s)t_{i-1} + s t_i}$, $s \in [0,1]$ (with respect to the naturally reparametrized filtration), obtaining the inequality
\begin{align*}
 & \sum_{t_i \in \pi}  \abs{ \bra{   \E\left[ \cc\bra{ \frac{X_{t_i}-X_{t_{i-1}}}{\sqrt{t_i-t_{i-1}}}} \right]- \E\left[  \cc\bra{\textstyle{\sqrt{\dot{\qv{X}} }_{t_{i-1}}}\, Z}  \right]} } (t_i- t_{i-1})  \\
\leq  &  \sum_{t_i \in \pi}  \eps \E  \sqa{ (t_i-t_{i-1})(1+ |Z|^{2p}) + \int_{t_{i-1}}^{t_i} |\sigma_s|^{2p} \d s  } +\lambda(\eps) \E  \sqa{  \int_{t_{i-1}}^{t_i} |\sigma_s - \sigma_{t_{i-1}} |^{2p} \d s  }
\end{align*}
Letting $\norm{\pi} \to 0$, we obtain by continuity of $\sigma$ that
\[\begin{split}  \limsup_{ \norm{\pi} \to 0} \sum_{t_i \in \pi}  & \abs{ \bra{   \E\left[  \cc\bra{ \frac{X_{t_i}-X_{t_{i-1}}}{\sqrt{t_i-t_{i-1}}}}\right]- \E\left[  \cc\bra{\textstyle{\sqrt{\dot{\qv{X}} }_{t_{i-1}}}\, Z}  \right]} } (t_i- t_{i-1}) \\
&\le    \eps  \E\sqa{ 1+ |Z|^{2p} + \int_0^1 |\sigma_s|^{2p} \d s},\end{split}\]
and as $\eps \to 0$ we conclude that \eqref{eq:proof-relaxation-error} holds.

\bibliographystyle{alpha}

\bibliography{joint_biblio}

\end{document}